\documentclass[11pt, oneside]{amsart}

\usepackage{amssymb}
\usepackage{amsmath}
\usepackage{amsthm}
\usepackage{dsfont}
\usepackage{enumerate}
\usepackage{enumitem}
\usepackage{multirow}
\usepackage{MnSymbol}
\usepackage{semmac}
\usepackage{longtable}
\usepackage{graphicx}

\theoremstyle{plain}
\newtheorem{theorem}{Theorem}[section]
\newtheorem{lemma}[theorem]{Lemma}
\newtheorem{proposition}[theorem]{Proposition}
\newtheorem{corollary}[theorem]{Corollary}

\newtheorem{definition}[theorem]{Definition}

\newtheorem{remark}[theorem]{Remark}
\newtheorem{construction}[theorem]{Construction}
\newtheorem*{notation}{Notation}
\newtheorem*{claim}{Claim}

\theoremstyle{remark}

\newtheorem{example}[theorem]{Example}

\title{Clusters, inertia, and root numbers}
\author{Matthew Bisatt}
\address{Fry Building, University of Bristol, Bristol, BS8 1UG, UK}
\email{matthew.bisatt@bristol.ac.uk}

\date{\today}

\begin{document}
\global\long\def\qq{\mathbb{Q}}
\global\long\def\qp{\mathbb{Q}_p}
\global\long\def\zz{\mathbb{Z}}
\global\long\def\cc{\mathbb{C}}
\global\long\def\cR{\mathcal{R}}
\global\long\def\ss{\mathfrak{s}}
\global\long\def\ds{d_{\ss}}
\global\long\def\Is{I_{\ss}}
\global\long\def\es{\varepsilon_{\ss}}
\global\long\def\lcm{\operatorname{lcm}}
\global\long\def\sp{\operatorname{sp}}
\global\long\def\Gal{\operatorname{Gal}}
\global\long\def\Jac{\operatorname{Jac}}
\global\long\def\Stab{\operatorname{Stab}}
\global\long\def\Sym{\operatorname{Sym}}
\global\long\def\ord{\operatorname{ord}}
\global\long\def\Ind{\operatorname{Ind}}
\global\long\def\minpoly{\operatorname{minpoly}}
\global\long\def\denom{\operatorname{denom}}

\begin{abstract}
	In a recent paper of Dokchitser--Dokchitser--Maistret--Morgan, the authors introduced the concept of a cluster picture associated to a hyperelliptic curve from which they are able to recover numerous invariants, including the inertia representation on the first \'{e}tale cohomology group of the curve. The purpose of this paper is to explore the functionality of these cluster pictures and prove that the inertia representation of a hyperelliptic curve is a function of its cluster picture.
\end{abstract}

\maketitle
\tableofcontents

\newpage
\section{Introduction}

Given a hyperelliptic curve $C: y^2=f(x)$ over a non-Archimedean local field $K$, one is generally interested in both its arithmetic and geometric properties such as:
\begin{itemize}
	\item Is it semistable?
	\item What is the special fibre of the minimal regular model?
	\item What is the conductor of its Jacobian?
	\item What is the Galois action on its first \'{e}tale cohomology?
\end{itemize}

It turns out that these properties are intimately related to the configuration of the roots of $f$ over an algebraic closure via the associated valuation; such a configuration is called a cluster picture. In particular, if the residue field of $K$ has odd characteristic, then Dokchitser--Dokchitser--Maistret--Morgan have recently shown \cite{DDMM} that the above questions can be approached via these objects. 

Cluster pictures are however a relatively new invention in this context. Currently, the approach to the above questions still relies on knowledge of the original polynomial $f$, in particular the inertia action on its roots, in order to answer the above questions. Our purpose is to remove this dependency on the polynomial and to explore their functionality in more detail.

In particular, we provide answers to the following questions when the splitting field of $f$ is tamely ramified.
\begin{enumerate}
	\item Can the inertia action on the roots be recovered from the cluster picture?
	\item To what extent does the inertia action and cluster configuration determine the distances between roots?
	\item When does an abstract cluster picture arise from a hyperelliptic curve?
	\begin{enumerate}
		\item If so, can we give an example of such a curve?
	\end{enumerate}
	\item Can one give a closed form for the inertia representation of $C$?
	\begin{enumerate}
		\item If so, can it be simplified to compute the root number?
	\end{enumerate}
		\item Does the cluster picture of an elliptic curve determine its Kodaira type or root number?
\end{enumerate}

This increased functionality will make the formula for the inertia representation more amenable for families of hyperelliptic curves and in particular to study the distribution of root numbers across this family, building on the author's previous work \cite{Bis19}. As an example, we determine all inertia representations of genus two hyperelliptic curves from their cluster classification in Appendix \ref{tables}; this is similar to the more established Namikawa--Ueno classification \cite{NU73} but has the benefit of being simpler to use. Moreover, we show that the cluster picture also encodes the classical Kodaira type for elliptic curves.

In the course of this paper, we approach each of our questions and focus more on the applications than the theoretical aspect as this is how we envision our results being used in future work. To this extent, we have relegated the bulk of the proofs to the appendices, and demonstrate their applications through an extended example in \S \ref{egsection}.

\subsection{Statement of results}
We shall quickly introduce some notation and the relevant definitions to state our main theorems. \\

\noindent \textbf{Notation.} \\
\begin{tabular}{cl}
	$K$ & finite extension of $\qp$, $p$ odd; \\
	$\overline{K}$ & algebraic closure of $K$; \\
	$\mathcal{O}_K$ & ring of integers; \\
	$v$ & valuation of $\overline{K}$ such that $v(K^{\times})=\zz$; \\
	$f$ & $\in K[x]$ a squarefree polynomial of degree $\deg f$; \\
	$\cR$ & $=\{r \in \overline{K} \, | \, r \text{ is a root of }f \}$; \\
	$I$ & inertia subgroup of $\Gal(K(\cR)/K)$.
\end{tabular}

\begin{definition}
	Let $\alpha=\frac{a}{b} \in \qq$ with $\gcd(a,b)=1$ and $b>0$. Then we define the denominator of $\alpha$ as $$\denom \alpha =b.$$
\end{definition}

\begin{definition}
\label{clusterdef}
\begin{itemize}
	\item[]
	\item A \emph{cluster} is a nonempty subset $\ss \subseteq \cR$ of the form $\ss = D \cap \cR$ for some disc $D= \{x \in \overline{K} \, | \, v(x-z) \geqslant d\}$, for some $z \in \overline{K}$ and $d \in \qq$.
	\item If $|\ss|>1$, then $\ss$ is called \emph{proper} and we associate to it a depth $$\ds:= \min_{r,r' \in \ss} v(r-r').$$
	\item If $\ss' \subsetneq \ss$ is a maximal subcluster, then we call $\ss'$ a child of $\ss$. We further call $\ss$ the parent of $\ss'$ and write this as $\ss=P(\ss')$.
	\item For a cluster $\ss$, we denote by $\Is$ the stabiliser of $\ss$ under $I$.
	\item Fix a cluster $\ss$. If there exists a unique fixed child $\ss'$ under the action of $\Is$, then we call $\ss'$ an \emph{orphan} of $\ss$.
	\item The cluster picture of a squarefree polynomial is the collection of all clusters of its roots.
\end{itemize}
\end{definition}

We can now state our first main result which concerns the functionality of cluster pictures; the proof of this can be found in the appendices.

\begin{theorem}[=Theorem \ref{mainpf}]
\label{main}
	Assume the inertia group $I$ is tame.
	\begin{enumerate}
		\item[]
		\item $\lcm_{\ss} \denom \ds = |I|$, where the $\lcm$ runs over all proper clusters.
		\item Fix a cluster $\ss$. Then the orbits of non-orphans of $\ss$ under $\Is$ all have equal length.
		\item Let $\ss'$ be a child of $\ss$ which is not an orphan. Then the length of its orbit under $\Is$ is $\denom(\ds [I:\Is])$.
		\item Let $\ss$ be a cluster. Then $$[I:\Is] = \lcm_{\ss \subsetneq \ss'} \denom d_{\ss'}^*,$$ where for a cluster $\ss' \supsetneq \ss$, $$d_{\ss'}^*=\begin{cases}
		1 &\text{ if the child of $\ss'$ containing $\ss$ is an orphan}, \\
		d_{\ss'} &\text{ else}.
	\end{cases}$$
	\end{enumerate}
\end{theorem}

\begin{remark}
	Note that $d_{\ss'}^*$ depends on the cluster $\ss$ and not just on $\ss'$.
\end{remark}

\begin{corollary}
	Assume the inertia group $I$ is tame. Then the cluster picture (with depths) of a polynomial $f$ determines $H^1_{\acute{e}t}(C/\overline{K},\qq_{\ell})$ as an $I$-representation where $C/K: y^2=f(x)$ is the corresponding hyperelliptic curve; the representation is explicitly given in Theorem \ref{repthm}.
\end{corollary}

As an example of their versatility, we show that the cluster picture of an elliptic curve also determines their Kodaira type.

\begin{theorem}[=Theorem \ref{kodthm}]
	Let $E/K:y^2=f(x)$, $\deg f=3$, be an elliptic curve and suppose that $p \geqslant 5$. Then the cluster picture of $f$ uniquely determines the Kodaira type of $E/K$.
\end{theorem}

Cluster pictures can be defined more abstractly in a combinatorial fashion and hence one can ask which cluster pictures arise from polynomials, and moreover whether one can construct an example of such a polynomial. This question is critical if one wishes for an explicit classification result for hyperelliptic curves via cluster pictures, in addition to enabling the study of distributions of local invariants. We are able to prove that the conditions imposed by Theorem \ref{main} are generally sufficient in the tame case for the cluster picture to arise from a polynomial; we provide the details in \S \ref{polyS}.

\begin{theorem}[=Theorem \ref{polytype}]
	Let $\Sigma$ be a cluster picture and suppose there is a cyclic automorphism of $\Sigma$ which satisfies Theorem \ref{main}i-iv. Then there exists a polynomial $f \in K[x]$ such that $\Sigma$ is isomorphic to the cluster picture associated to $f$ for $p$ sufficiently large. Moreover, such an $f$ can be constructed explicitly via Construction \ref{polycons}.
\end{theorem}

\noindent \textbf{Acknowledgements.} The author would like to thank Vladimir Dokchitser and Alex Betts for helpful discussions and the Max Planck Institute in Bonn for a wonderful working environment.

\section{Cluster pictures of polynomial type}
\label{polyS}

Cluster pictures can be defined abstractly in a purely combinatorial manner. Given a squarefree polynomial, we have already seen how to construct the corresponding cluster picture. This affords a map $$\psi: \{ \text{squarefree polynomials over }K \} \rightarrow \{ \text{cluster pictures} \}.$$ Note that $\psi$ is not injective since $\psi(f(x))=\psi(f(x-k))=\psi(kf(x))$ for all $k \in K^{\times}$. Instead, we concern ourselves with the image of $\psi$; there are two questions we would like to consider in this section:

\begin{enumerate}
\item When is a cluster picture in the image of $\psi$?
\item If a cluster picture $\Sigma$ is in the image of $\psi$, can we construct a polynomial in the preimage of $\Sigma$?
\end{enumerate}

To facilitate our answers to these questions, we should first define what a cluster picture is in the general combinatorial sense.

\begin{definition}
Let $X$ be a finite set, $\Sigma$ a collection of nonempty subsets of $X$; elements of $\Sigma$ are called \emph{clusters}. Attach a depth $d_{\ss} \in \mathbb{Q}$ to each cluster $\ss$ with $|\ss|>1$. Then $\Sigma$ (or rather $(\Sigma,X,d)$) is a \emph{cluster picture} if
\begin{enumerate}
\item Every singleton (``root'') is a cluster, as is $X$.
\item Two clusters are either disjoint or one is contained in the other.
\item $d_{\mathfrak{t}}>d_{\ss}$ is $\mathfrak{t} \subsetneq \ss$.
\end{enumerate}
Moreover, two cluster pictures $(\Sigma,X,d)$ and $(\Sigma',X',d')$ are isomorphic if there is a bijection $\phi:X \rightarrow X'$ which induces a bijection $\Sigma \rightarrow \Sigma'$ with $d_{\ss} = d'_{\phi(\ss)}$.

If $f$ is a polynomial, then the cluster picture $\Sigma_f$ is the collection of all clusters of roots of $f$, as in Definition \ref{clusterdef}.
\end{definition}

\begin{definition}
	Let $\Sigma$ be a cluster picture. Then $\Sigma$ is of \emph{polynomial type over $K$} if there exists a squarefree polynomial $f \in K[x]$ such that $\Sigma$ is isomorphic to $\Sigma_f$.
\end{definition}

To simplify our problem, we shall assume that there is only a ``tame action'' on the cluster picture; i.e. only answer our question when $\Sigma \cong \Sigma_f$, with $f$ a polynomial whose splitting field is tamely ramified. We write $\Sym(X)$ for the symmetric group with underlying set $X$. \\

\noindent \textbf{Hypothesis H.} Let $(\Sigma,X,d)$ be a cluster picture. We say $\Sigma$ satisfies Hypothesis H if there exists $c \in \Sym(X)$ which induces an automorphism of $\Sigma$ such that:
\begin{itemize}
	\item The orbits of non-orphan children of a proper cluster $\ss$ under $c$ all have length equal to $\denom(\ds [C:\Stab_C(\ss)])$ under $\Stab_C(\ss)$, where $C=\langle c \rangle$ is the subgroup generated by $c$;
	\item Let $\ss \in \Sigma$. Then $$[C:\Stab_C(\ss)]=\lcm_{\ss \subsetneq \ss'} \denom d_{\ss'}^*$$ where for a cluster $\ss' \supsetneq \ss$, $$d_{\ss'}^*=\begin{cases}
		1 &\text{ if the child of $\ss'$ containing $\ss$ is an orphan}, \\
		d_{\ss'} &\text{ else}.
	\end{cases}$$
\end{itemize}

\begin{remark}
	Note that the above conditions imply that $|C|=\lcm_{\ss \in \Sigma} \ds$ (where the $\lcm$ runs over all proper clusters); this provides a useful preliminary criterion to check if a given cluster picture satisfies Hypothesis H.
\end{remark}

\begin{theorem}
\label{polytype}
	Let $(\Sigma,X,d)$ be a cluster picture and suppose $p>|X|$. Then $\Sigma$ is of polynomial type over $K$ if and only if $\Sigma$ satisfies Hypothesis H.
\end{theorem}

One direction of this theorem is obvious: if $\Sigma$ is of polynomial type, then $C=\langle c \rangle$ can be taken to be the inertia group of the splitting field of $f$ by virtue of Theorem \ref{main}. For the converse, we give an explicit construction of such a polynomial from the orphan cluster picture\footnote{This is a cluster picture with all orphans identified.} and prove that the cluster pictures are isomorphic. Note that identification of orphans is determined at the same time as establishing the automorphism $c$ in Hypothesis H.

\begin{construction}
\label{polycons}
	Let $(\Sigma,X,d)$ be a cluster picture which satisfies Hypothesis H and suppose $p>|X|$. Choose a set of representatives $Y \subset X$ for $X/C$ such that for all clusters $\ss \in \Sigma$, if $\ss \cap Y \neq \emptyset$ then $(c^k \cdot \ss) \cap Y = \emptyset$ whenever $c^k \cdot \ss \neq \ss$. For each $y \in Y$ and proper cluster $\ss \in \Sigma$, choose $a(y,\ss) \in \mathcal{O}_K^{\times} \cup \{0\}$ such that:
	\begin{enumerate}
	\item $a(y,\ss)=0$ if $y \not\in \ss$;
	\item Let $y_1 \neq y_2 \in Y$ and let $\ss'$ be the smallest cluster containing both $y_1$ and $y_2$. Then
	\begin{enumerate}
		\item $a(y_1,\ss)=a(y_2, \ss)$ for all $\ss \supseteq P(\ss')$;
		\item $v(a(y_1,\ss')-\zeta_e^k a(y_2,\ss'))=0$ whenever $[C:\Stab_C(\ss')] \, | \, k$;
	\end{enumerate}
	\item If $y \in \ss$, then $a(y,P(\ss))=0 \Leftrightarrow \ss$ is an orphan of $P(\ss)$.
	\end{enumerate}
	Then we define the polynomial $$f_Y = \prod\limits_{y \in Y} \minpoly_K (\alpha(y)),$$ where $\alpha(y)=\sum\limits_{\ss \text{ proper}} a(y,\ss)\pi_K^{\ds}$ and $\minpoly$ denotes the minimal polynomial.
\end{construction}

\begin{remark}
	Note that in checking criterion (ii)(b) of the $a(y,\ss)$ in the construction above, we can further reduce to checking the roots $\zeta_e^k$ which are in $K$. This is because $K(\zeta_e)/K$ is unramified so if $\zeta_e^k \not\in K$, then $\zeta_e^k a(y_2,\ss')$ is in the residue field of $K(\zeta_e)$ but not $K$ and hence cannot have the same residue as $a(y_1,\ss')$.
\end{remark}

\begin{remark}
	Our construction builds a polynomial with integral roots and hence all the depths should be non-negative. This is not a restriction though: there exists an integer $m$ such $\ds+m \geqslant 0$ for all proper clusters $\ss$ so we can construct a polynomial $f(x) \in \mathcal{O}_K[x]$ whose cluster picture has these new depths; the polynomial $f(\pi_K^{-m}x) \in K[x]$ then has the correct depths.
\end{remark}

We shall first prove a weaker form of Theorem \ref{polytype}, where we impose different conditions on the construction.

\begin{lemma}
\label{polylem}
	Suppose $\Sigma$ is a cluster picture satisfying Hypothesis H where the automorphism $C$ has order $e$ and suppose $p>|X|$. Fix a primitive $e^{\text{th}}$ root of unity $\zeta_e$ and suppose $\zeta_e \in K$. For each $r \in X$ and proper cluster $\ss \in \Sigma$, choose $a(r,\ss) \in \mathcal{O}_K^{\times} \cup \{0\}$ subject to the following constraints:
	\begin{enumerate}
	\item $a(r,\ss)=0$ if $r \not\in \ss$;
	\item Let $r_1 \neq r_2 \in X$ and let $r_1 \wedge r_2$ be the smallest cluster containing both $r_1$ and $r_2$. Then
	\begin{enumerate}
		\item $a(r_1,\ss)=a(r_2, \ss)$ for all $\ss \supseteq P(r_1 \wedge r_2)$;
		\item $a(r_1,r_1 \wedge r_2)-a(r_2,r_1 \wedge r_2) \in \mathcal{O}_K^{\times}$;
	\end{enumerate}
	\item If $r \in \ss$, then $a(r,P(\ss))=0 \Leftrightarrow \ss$ is an orphan of $P(\ss)$;
	\item $a(c^k \cdot r, c^k \cdot \ss)= \zeta_e^{ke\ds}a(r,\ss)$ for all $k \in \zz$.
	\end{enumerate}
	
	Define $$\alpha(r)=\sum\limits_{\ss \text{ proper}} a(r,\ss)\pi_K^{\ds},$$ (note that this is an injective map into $\overline{K}$ by construction) and set $$f:=\prod\limits_{r \in X} (x-\alpha(r)).$$ Then $\Sigma \cong \Sigma_f$.
\end{lemma}

\begin{remark}
\label{primecon}
	We use the assumption $p>|X|$ which ensures that the residue field is sufficiently large for constraint (ii)(b), as well as guaranteeing that the action is ``tame''; for an individual cluster picture, this may be possible for smaller $p$.
\end{remark}

\begin{proof}
	To ensure that this construction is well-defined, we should check that if $c^k \cdot r= r$, then $a(c^k \cdot r, c^k \cdot \ss)=a(r,\ss)$ for all $\ss$, i.e. that $k\ds \in \zz$ whenever $r \in \ss$ is such that the child of $\ss$ containing $r$ is not an orphan. As $c^k \cdot r = r$, we necessarily have $[C:\Stab_C(r)] \, | \, k$; by assumption $[C:\Stab_C(r)]=\lcm_{\ss \ni r} \denom \ds^*$ and hence it is well-defined.
	
	Observe that $\alpha$ preserves distances between distinct elements of $X$ since $v(\alpha(r_1)-\alpha(r_2))=d_{r_1 \wedge r_2}$ by property (ii). As every cluster of $\Sigma$ is of the form $r_1 \wedge r_2$, this implies that $\alpha$ induces a depth-preserving bijection on the clusters of $\Sigma$ with the image of $r_1 \wedge r_2$ being $\alpha(r_1) \wedge \alpha(r_2)$. The surjection is due to the fact that $f$ has roots $\alpha(r_i)$ and hence every cluster of $\Sigma_f$ has the form $\alpha(r_1) \wedge \alpha(r_2)$.

	So far we have constructed an isomorphic cluster picture $\Sigma_f$, but we have yet to show that $f$ is actually defined over $K$ and that the inertia action on the roots of $f$ is given by the image of $C$ under $\alpha$. First note that $\alpha(r) \in K(\pi_K^{1/e})$ for all $r \in X$ and let $I=\Gal(K(\pi_K^{1/e})/K)$ be the inertia group. Then by property (iv), if $r_2=c^k \cdot r_1$ for some integer $k$, then $\alpha(r_2)=i^k(\alpha(r_1))$ where $i$ is the generator of $I$ such that $i(\pi_K^{1/e})=\zeta_e\pi_K^{1/e}$ and hence $f$ is invariant under $I$ so $f \in \mathcal{O}_K[x]$. In addition, the remaining properties imposed ensure that $[I:\Stab_I(\alpha(r))]=[C:\Stab_C(r)]$ for all $r$ and hence $I$ is the image of $C$ under $\alpha$.
\end{proof}

\begin{proof}[Proof of Theorem \ref{polytype}]
	First note that the cluster picture of a polynomial is invariant under unramified extensions of $K$. Moreover for all $y \in Y$, $K(\alpha(y))/K$ is totally ramified and hence its Galois conjugates are precisely the inertia conjugates. We may therefore assume $\zeta_e \in K$ and hence it suffices to prove that $f_Y$ satisfies the conditions of Lemma \ref{polylem}. Observe that by construction, $a(c^k \cdot y, c^k \cdot \ss)=i(\pi_K^{k\ds})a(y,\ss)= \zeta_e^{ke\ds}a(y,\ss)$ and hence satifies constraint (iv). Now let $\ss$ be a proper cluster and suppose $c^k \cdot y \not\in \ss$. Then $y \not\in c^{-k} \cdot \ss$ and hence $a(c^k \cdot y, \ss)=0$. Moreover if $\ss$ is an orphan, then so is $c^{-k} \cdot \ss$ and hence we have constraints (i) and (iii). It remains to check (ii).
	
	Let $\ss'$ be the smallest cluster containing $c^k \cdot y_1$ and $c^l \cdot y_2$, for some integers $k,l$. Note that we may assume $l=0$ by applying $c^{-l}$ to the roots and $\ss'$ without affecting the statement. By the construction of $Y$, we necessarily have $y_1$; if not, then $c^k \cdot \ss' \neq \ss'$ both have nontrivial intersection with $Y$ which is a contradiction. Moreover, this implies that $c^k \cdot \ss' = \ss'$ so therefore $[C:\Stab_C(\ss')] \, | \, k$ and hence $\zeta_e^{ke\ds}=1$ for every non-orphan cluster $\ss$ properly containing $\ss'$. Therefore $a(c^k \cdot y_1, \ss)=a(y_1,\ss)$ for all clusters $\ss$ properly containing $\ss'$ and (ii)(a) follows by the analogous condition on $y_1$ and $y_2$.
	
	We now must show that $a(c^k \cdot y_1,\ss')-a(y_2,\ss')$ is a unit in $K$. We know that $y_1 \in \ss'$ so we have two cases depending on whether $\ss'$ is the smallest cluster containing $y_1$ and $y_2$ or not. If so, then this follows from (ii)(b). Otherwise, $\ss'$ is the smallest cluster containing $y_1$ and $c^k \cdot y_1$; in this case it suffices to show $(1-\zeta_e^{ked_{\ss'}})$ is a unit. If $\zeta_e^{ked_{\ss'}} \neq 1$, then we are done since roots of unity necessarily have distinct images in the residue field. Now $\zeta_e^{ked_{\ss'}} = 1 \Leftrightarrow \denom d_{\ss'} | k$. This is impossible since it would imply that the child of $\ss'$ containing $y_1$ is stable under $c^k$, and hence contains $c^k \cdot y_1$, which is a contradiction to minimality. 
\end{proof}

\section{Recovering depth denominators}
Unfortunately, knowledge of the inertia action on roots (and hence all clusters), together with the collection of clusters, does not enable one to completely recover the denominators of the depths.

\begin{example}
Fix a prime $p \geqslant 5$ and primitive cube root of unity $\zeta_3$. Consider the polynomials $$f_1= x^6 -2p(3p+1)x^3 -p^2(p^3-3p^2+3p-1),$$ with roots $\zeta_3^j p^{1/3}+ (-\zeta_3)^j p^{5/6}, 0 \leqslant j \leqslant 5,$ and $$f_2= x^6 -3px^4 -2px^3+3p^2x^2-6p^2+(p^2-p^3),$$ with roots $\zeta_3^j p^{1/3} + (-1)^j p^{1/2}, 0 \leqslant j \leqslant 5$.

Both polynomials have $C_6$ inertia action on their roots and the unlabelled cluster picture is
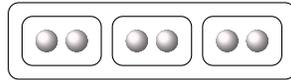
\begin{center}
\clusterpicture
\Root(0.00,2)A(a1);     
\Root(0.40,2)A(a2);     
\Root(1.20,2)A(a3);      
\Root(1.60,2)A(a4);
\Root(2.40,2)A(a5);
\Root(2.80,2)A(a6);
\ClusterL(c1){(a1)(a2)}{};
\ClusterL(c2){(a3)(a4)}{};
\ClusterL(c3){(a5)(a6)}{};
\ClusterL(c4){(c1)(c2)(c3)}{};
\endclusterpicture
\end{center}

\noindent with the top cluster $\cR$ having depth $\frac{1}{3}$ and permuting its three children. The inner clusters of size $2$ each permute the roots contained in them (under their $C_3$ stabiliser), but either have depth $\frac{5}{6}$ (in the case of $f_1$) or $\frac{1}{2}$ (in the case of $f_2$).
\end{example}

As the example above shows, we do not have a bijection between possible inertia actions and possible depth denominators for a given cluster picture. However, we can apply Theorem \ref{main} to ascertain all possible sets of depth denominators for a given inertia action; we demonstrate this through an example.


\begin{example}
Consider the following cluster picture with $$I = \langle (r_1,r_7,r_4,r_{10},r_2,r_8,r_5,r_{11},r_3,r_9,r_6,r_{12}) \rangle \cong C_{12},$$ where we order the roots from left to right.
	
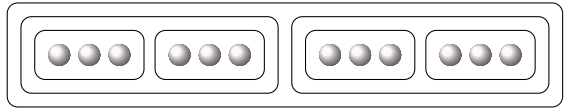
\begin{center}
\clusterpicture
\Root(0.00,2)A(a1);     
\Root(0.40,2)A(a2);     
\Root(0.80,2)A(a3);      
\Root(1.60,2)A(a4);
\Root(2.00,2)A(a5);
\Root(2.40,2)A(a6);
\Root(3.60,2)A(b1);
\Root(4.00,2)A(b2);
\Root(4.40,2)A(b3);
\Root(5.20,2)A(b4);
\Root(5.60,2)A(b5);
\Root(6.00,2)A(b6);
\ClusterL(s1){(a1)(a2)(a3)}{};
\ClusterL(s2){(a4)(a5)(a6)}{};
\ClusterL(s3){(s1)(s2)}{};
\ClusterL(s4){(b1)(b2)(b3)}{};
\ClusterL(s5){(b4)(b5)(b6)}{};
\ClusterL(s6){(s4)(s5)}{};
\ClusterL(s7){(s3)(s6)}{};
\endclusterpicture
\end{center}	
	
We have $7$ proper clusters for which to determine the depths: 
\begin{eqnarray*}
\ss_1 &=& \{ r_1,r_2,r_3 \}, \qquad \ss_2=\{r_4,r_5,r_6\}, \qquad \ss_3=\ss_1 \cup \ss_2, \\ \relax
\ss_4 &=& \{r_7,r_8,r_9\}, \qquad \ss_5=\{r_{10},r_{11},r_{12} \}, \qquad \ss_6=\ss_4 \cup \ss_5, \quad \cR.
\end{eqnarray*}

First we identify the inertia orbits on proper clusters to reduce our computation since inertia-conjugate clusters have the same depth (Remark \ref{Galval}); for example $d_{\ss_3}=d_{\ss_6}$. This reduces us to determining possible depth denominators for $\cR, \ss_3$ and $\ss_1$.

As before, we work from the outside in and hence start with $\cR$. Note that its children $\ss_3$ and $\ss_6$ have orbit length $2$ hence $2=\denom d_{\cR}$ by Theorem \ref{main}iii. We now consider $\ss_3$. This has two children which are permuted hence $2=\denom [I:I_{\ss_3}]d_{\ss_3} = \denom 2d_{\ss_3}$; the only choice is $\denom d_{\ss_3}=4$.

Finally consider $\ss_1$. This has three children which are all permuted under $I_{\ss_1} \cong C_4$ so have orbit length $3$. Hence, by Theorem \ref{main}iii, we have $3=\denom 4d_{\ss_1}$. This does not have a unique solution but three different possibilities: $\denom d_{\ss_1} \in \{3,6,12 \};$ moreover each option is compatible with Theorem \ref{main}i. Indeed each case occurs; since all the roots are necessarily Galois conjugate, we shall only give one root $\alpha$ of the degree $12$ polynomial over $\qp$ in the table below.

\begin{center}
\begin{tabular}{c|c|c|c}
	$\denom d_{\cR}$ & $\denom d_{\ss_3}$ & $\denom d_{\ss_1}$ & $\alpha$ \\ \hline
	$2$ & $4$ & $3$ & $p^{1/2}+p^{3/4}+p^{4/3}$ \\
	$2$ & $4$ & $6$ & $p^{1/2}+p^{3/4}+p^{5/6}$ \\
	$2$ & $4$ & $12$ & $p^{1/2}+p^{3/4}+p^{11/12}$ \\
\end{tabular}
\end{center}

\end{example}

\section{The inertia representation}
\label{rootnosection}

In this section, we are going to apply the results built up so far in order to determine both $H^1_{\acute{e}t}(C/\overline{K},\qq_{\ell})$ purely in terms of the information contained in its cluster picture. We deal with the hyperelliptic curve $C/K: y^2=f(x)$ and let $c_f$ be the leading coefficient of $f$. To state the relevant theorems, we will need to introduce more notation and define some more invariants associated to clusters; see \cite{DDMM} for the full definitions.

\begin{definition}
Let $\ss$ be a cluster. Then we call $\ss$ \emph{odd} (resp. \emph{even}) if $|\ss|$ is odd (resp. even). We say $\ss$ is \emph{\"{u}bereven} if all children are even. We further define $\ss$ to be a \emph{cotwin} if it is not \"{u}bereven and has a child of size $2g$, where $|\cR| \in \{2g+1,2g+2\}$.

If $\ss$ is not a cotwin, we define $\ss^*$ to be the smallest cluster $\ss^* \supseteq \ss$ whose parent is not \"{u}bereven (and $\ss^*=\cR$ if no such cluster exists); if $\ss$ is a cotwin, then we let $\ss^*$ be its child of size $2g$.

If $\ss$ is proper, then we attach the following quantities: \\
\begin{tabular}{l p{0.85\linewidth}}
$\ss^{odd}$ & set of all odd children of $\ss$; \\
$\lambda_{\ss}$ & $=\frac{1}{2}[I:\Is](v(c_f)+d_{\ss}|\ss^{odd}| + \sum\limits_{r \not\in \ss} v(r-r_0))$ for any $r_0 \in \ss$; \\
$\gamma_{\mathfrak{s}}$ & any character of $I_{\ss}$ of order equal to the prime-to-$p$ part of the denominator of $\lambda_{\ss}$ (with $\gamma_{\ss}=\mathds{1}$ if $\lambda_{\ss}=0$).
\end{tabular}

If $\ss$ is even or a cotwin, we define $\es : \Is \rightarrow \{ \pm 1 \}$ to be $$\es(\sigma) \equiv \frac{\sigma(\theta_{\ss^*})}{\theta_{\ss^*}} \bmod{\mathfrak{m}},$$ where $\theta_{\ss}=\sqrt{c_f \prod\limits_{r \not\in \ss} (r_0 -r)}$ for any $r_0 \in \ss$ and $\mathfrak{m}$ the maximal ideal of algebraic closure of $K$. For other clusters $\ss$, we set $\es$ to be the zero character.

Lastly, we set $V_{\mathfrak{s}}=\gamma_{\mathfrak{s}} \otimes (\cc[\ss^{odd}] \ominus \mathds{1}) \ominus \epsilon_{\mathfrak{s}}$.

\end{definition}

\begin{remark}
	In \cite{DDMM}, the authors define $\es$ as a map $G_K \rightarrow \{\pm 1 \}$. The definition of $\es$ we give here is equal to the restriction of their $\es$ to the inertia stabiliser of $\ss$, where it is a character and independent of all choices.
\end{remark}

\begin{theorem}[{\cite[Theorem 1.19]{DDMM}}]
\label{M2D2rep}
	Let $\ell \neq p$ be prime. Then, as representations of the absolute inertia group of $K$: $$H^1_{\acute{e}t}(C/\overline{K},\qq_{\ell}) \otimes \mathbb{C} \cong H^1_{ab} \oplus (H^1_t \otimes \sp(2) ),$$ where $$H^1_{ab} = \bigoplus\limits_{\ss \in X/I} \Ind V_{\ss}, \qquad H^1_t = \left( \bigoplus\limits_{\ss \in X/I} \Ind \varepsilon_{\ss} \right) \ominus \varepsilon_{\cR},$$ where $X$ is the set of proper non-\"{u}bereven clusters.
\end{theorem}

To describe this representation in a closed form, we shall explicitly describe $\Ind V_{\ss}$ for a fixed cluster $\ss$ as well as $\Ind \es$; formulae for these will then completely recover both $H^1_{ab}$ and $H^1_t$ as above.

\begin{definition}
	Let $n,d$ be positive integers. Then we define $$\gcd(n,d^{\infty}):= \lim\limits_{k \to \infty} \gcd(n,d^k).$$
\end{definition}
	
\begin{remark}
	Note that for a prime $q$, $q| \gcd(n,d) \Leftrightarrow q| \gcd(n,d^{\infty})$ and if so, then $v_q(\gcd(n,d^{\infty}))=v_q(n)$. Equivalently, $\gcd(n,d^{\infty})$ is the minimal positive divisor $g$ of $n$ such that $\gcd(\frac{n}{g},d)=1$.
\end{remark}

\begin{notation}
For a cluster $\ss$, we define
\begin{enumerate}
\item $n_{\ss}=[I:\Is]$;
\item $n'_{\ss}=\denom (\ds n_{\ss})$, the orbit length of non-orphan children of $\ss$ under $\Is$.
\end{enumerate}
\end{notation}

\begin{notation}
For positive integers $d,t$, we define the following:

\begin{tabular}{ll}
$\varphi(d)$ & $=|(\zz/d\zz)^{\times}|$; \\
$\rho_d$ & direct sum of all characters of order $d$ of $\cc[C_d]$; \\
$A_{d,t}$ & $= \{ q \text{ prime} \, | \, v_q(d)=v_q(t)>0 \}$; \\
$S_{d,t}$ &  $=\left\lbrace \left. \dfrac{\lcm(d,t)}{\prod\limits_{q_i \in A_{d,t}} q_i^{m_i}} \, \right| \, 0 \leqslant m_i \leqslant v_{q_i}(t) \right\rbrace;$ \\
$Q_{q,k}(s)$ & $=\begin{cases}
	\frac{q-2}{q-1} &\text{ if } v_q(s)=k \qquad \text{ where $q$ is prime}, \\
	1 &\text{ else }. \end{cases}$ \\
$\beta(n,s)$ & $=\frac{\gcd(n,s^{\infty})\varphi(s)}{\varphi(\gcd(n,s^{\infty})s)}.$
\end{tabular}
\vspace{10pt}

If $s \in S_{d,t}$, then we further define

\begin{tabular}{ll}
$\alpha_{d,t,s}$ & $=\dfrac{\varphi(d)}{\varphi(\lcm(d,t))}\prod\limits_{q_i \in A_{d,t}} Q_{q_i,v_{q_i}(t)}(s).$
\end{tabular}

In the theorem below, for $a \in \qq$, we shall write $a\rho_d$ to mean $\rho_d^{\oplus a}$. We note that whilst $a$ will not necessarily be an integer, it will be true that $\denom a| \varphi(d)$ (so one may consider $a\rho_d$ as the direct sum of $a\varphi(d)$ distinct characters of order $d$) and also that $\Ind_{\Is}^{I} V_{\ss}$ will be the direct sum of \emph{integer} multiples of $\rho_d$. 

\end{notation}

\begin{theorem}
\label{repthm}
	Let $C/K: y^2=f(x)$ be a hyperelliptic curve with tame reduction. Let $\ss$ be a cluster and suppose that $\gamma_{\ss}$ has order $t$. Then 
	\begin{eqnarray*}
		\Ind_{\Is}^I V_{\ss} &=&  \left\lfloor \dfrac{|\ss^{odd}|}{n'_{\ss}} \right\rfloor \bigoplus\limits_{d|n'_{\ss}} \bigoplus\limits_{s \in S_{d,t}} \bigoplus\limits_{n_1|\frac{n_{\ss}}{\gcd(n_{\ss},s^{\infty})}} \alpha_{d,t,s} \beta(n_{\ss},s) \rho_{s\gcd(n_{\ss},s^{\infty})n_1} \\
		&\oplus & \left( |\ss^{odd}| - n'_{\ss}\left\lfloor \dfrac{|\ss^{odd}|}{n'_{\ss}} \right\rfloor -1 \right) \frac{\beta(n_{\ss},t)}{\varphi(t)} \bigoplus\limits_{n_2|\frac{n_{\ss}}{\gcd(n_{\ss},t^{\infty})}} \rho_{tn_2\gcd(n_{\ss},t^{\infty})} \\
		&\ominus & \Ind_{\Is}^I \es, 
	\end{eqnarray*}	
	where 
	
	$$\Ind_{\Is}^I \es = \begin{cases}
	0 &\text{if $\es=0$}; \\
	\bigoplus\limits_{m|n_{\ss}} \rho_m &\text{ if $\es$ is trivial}; \\
	\bigoplus\limits_{\substack{m|2n_{\ss} \\ m\nmid n_{\ss}}} \rho_m &\text{ if $\es$ has order $2$}.
	\end{cases}$$
\end{theorem}

\subsection{Simplification for root numbers}

We write $\chi_e$ for a ramified character of order $e$ on the full inertia group $I$.

\begin{theorem}[{\cite[Theorem 1.4]{Bis19}}]
\label{rootthm}
Let $K$ be a non-Archimedean local field and let $q$ be the cardinality of its residue field. Let $A/K$ be an abelian variety with tame reduction and write $\rho_A = \rho_B \oplus (\rho_T \otimes \sp(2))$, where $\rho_T$ is an Artin representation and $\rho_A$ is a Weil representation with finite image of inertia. Let $m_T = \langle \rho_T \mid_I, \chi_2 \rangle$ and $m_e= \langle \rho_B \mid_I, \chi_e \rangle \text{ for } e\geqslant 2.$ Then $$W(A/K) = \left( \prod\limits_{e \geqslant 3} W_{q,e}^{m_e} \right) (-1)^{\langle \rho_T, \mathds{1} \rangle} W_{q,2}^{m_T + \frac{1}{2}m_2},$$ where for an integer $k>0$ and rational odd prime $l$:
	
$W_{q,e} = \begin{cases}
	\left( \dfrac{q}{l} \right) &\text{ if } e=l^k; \\[10pt]	
	\left( \dfrac{-1}{q} \right) &\text{ if } e=2l^k \quad \text{ and } l \equiv 3 \mod{4} \qquad \text{ or } e=2; \\[10pt]	
	\left( \dfrac{-2}{q} \right) &\text{ if } e=4; \\[10pt]	
	\left( \dfrac{2}{q} \right) &\text{ if } e=2^k \quad \text{ for } k \geqslant 3; \\[10pt]	
	\quad 1 &\text{ else.}
\end{cases}$
\end{theorem}

We shall further examine the contribution to the root number from the toric part, i.e. the $\varepsilon$ terms. As mentioned previously, the depths do not encode the action of Frobenius so we are unable to recover $\langle \rho_T, \mathds{1} \rangle$ to distinguish between, for example, split and nonsplit multiplicative reduction. All that remains is to determine $m_T$, which counts the number of order $2$ characters in $\rho_T$.

\begin{proposition}
\label{toricprop}
	Assume $I$ is cyclic. Let $m_T(\ss)$ denote the number of quadratic characters in $\Ind_{\Is}^I \es$ and let $n_{\ss}=[I:\Is]$. Then $m_T(\ss) \leqslant 1$ and moreover:
	\begin{enumerate}
		\item If $\es=0$, then $m_T(\ss)=0$;
		\item If $\es=\mathds{1}$, then $m_T(\ss) \not\equiv n_{\ss} \bmod{2}$;
		\item If $\es$ has order $2$, then $m_T(\ss) \equiv n_{\ss} \bmod{2}$.
	\end{enumerate}
\end{proposition}

\begin{proof}
	Since $I$ is cyclic, $m_T(\ss) \leqslant 1$ necessarily. Now $\Ind_{\Is}^I \es$ is either $0$, $\cc[C_{n_{\ss}}]$ or $\cc[C_{2n_{\ss}}] \ominus \cc[C_{n_{\ss}}]$ respectively, from which the proposition follows.
\end{proof}

\begin{corollary}
	Let $X^{ev}$ denote the set of proper even, non-\"{u}bereven clusters and let $\ord_2$ denote the $2$-adic valuation. Let $\mu =1+ \ord \varepsilon_{\cR}$ if $\cR$ is even or a cotwin and $0$ otherwise, where $\ord \varepsilon_{\cR}$ is the order of $\varepsilon_{\cR}$ as a character of $I$. Then $$m_T \equiv \mu + \# \left\lbrace  \ss \in X^{ev}/I: \ord_2\left(n_{\ss}(v(c_f)+\sum\limits_{r \not\in \ss}d_{\{r\} \wedge \ss}) \right) \geq 1 \right\rbrace  + \sum\limits_{\ss \in X^{ev}/I} \!\!\! n_{\ss}  \quad \bmod{2},$$ where $\{r\} \wedge \ss$ is the smallest cluster containing $\{r\}$ and $\ss$.
\end{corollary}

\begin{proof}
Note that $m_T=-m_T(\cR) + \sum\limits_{\ss \in X^{ev}/I} m_T(\ss)$ and hence $m_T \equiv m_T(\cR) + |\{\ss \in X^{ev}/I: \es=\mathds{1} \}| + \sum\limits_{\ss \in X^{ev}/I} n_{\ss} \bmod{2}$ by Proposition \ref{toricprop}. If $\varepsilon_{\cR}$ is not the zero character (i.e. $\ss$ is even or a cotwin), then $m_T(\cR)=1$ if $\varepsilon_{\cR}$ is quadratic and $m_T(\cR)=1$ if $\varepsilon_{\cR}$ is trivial on inertia, hence $m_T(\cR) \equiv \mu \bmod{2}$.

The statement now follows since for $\ss$ in $X^{ev}$, $\es=\mathds{1}$ is equivalent to $\ord_2\left(n_{\ss}(v(c_f)+\sum\limits_{r \not\in \ss}d_{\{r\} \wedge \ss}) \right) \geq 1$ by \cite[Remark 1.13]{DDMM}.
\end{proof}

\begin{remark}
	Note that $\mu \equiv v(c_f) \bmod{2}$ if $\cR$ is even and $\mu=0$ if $\cR$ is odd and not a cotwin. If $\cR$ is an odd cotwin, then one can still determine $\cR^*$ from the cluster picture and hence determine $\mu$ as well, but this does not have such a compact formulation.
\end{remark}

\section{Elliptic curves}

In this section, we return to the more familiar territory of elliptic curves. We will use the cluster pictures to recover the inertia representation of $E$, along with its Kodaira type. Ideally, we would apply Theorem \ref{M2D2rep} for this purpose, except it is only stated in \cite{DDMM} for curves of genus at least $2$. Their proof should generalise to this situation,\footnote{The reason for this assumption is to avoid genus one curves without a rational point.} but we instead choose to prove it directly with a more elementary approach. For these reasons, we shall only consider elliptic curves in the form $$y^2= f(x),$$ where $f$ is a cubic. In this setting, there are precisely two cluster configurations of the roots, namely:
\vspace*{5pt}

\noindent\begin{minipage}{0.45\linewidth}
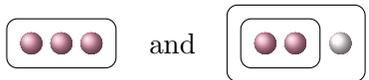
\begin{flushright}
\clusterpicture
\Root(0.00,2)B(a1);     
\Root(0.40,2)B(a2);     
\Root(0.80,2)B(a3);      
\ClusterL(c1){(a1)(a2)(a3)}{};
\endclusterpicture
\end{flushright}
\end{minipage}
\begin{minipage}{0.1\linewidth}
\begin{center}
and
\end{center}
\end{minipage}
\begin{minipage}{0.45\linewidth}
\clusterpicture
\Root(0.00,2)B(a1);     
\Root(0.40,2)B(a2);     
\Root(1.00,2)A(a3);      
\ClusterL(c1){(a1)(a2)}{};
\ClusterL(c3){(c1)(a3)}{};
\endclusterpicture .
\end{minipage}

\vspace*{5pt}

These two possibilities of cluster configurations happen to distinguish between potentially good and potentially multiplicative reduction, as we shall now prove.

\begin{lemma}
\label{potred}
	Let $E/K:y^2=f(x)$ be an elliptic curve and suppose that $f$ is monic. Moreover, assume that $p \neq 2$. Then:
	\begin{enumerate}
		\item The cluster picture (without depths) of $E$ is independent of the choice of model;
		\item $E$ has potentially good reduction if and only if the roots of $f$ are equidistant.
	\end{enumerate}
\end{lemma}

\begin{proof}
i) Observe that any other model for $E$ of the form $y_1^2=f_1(x_1)$ is obtained via a substitution of the form $y_1=u^3y, x_1=u^2x+s$, $u,s \in K, u \neq 0$. This changes the $p$-adic distances between the roots by a factor of $v(u^2)$, hence preserves whether or not the roots are equidistant.

ii) First note that $E$ has good reduction if and only if there exists a model of $E$ where the roots are integral and the discriminant $\Delta_E$ is a unit. As $p \neq 2$, $v(\Delta_E)=2(v(r_1-r_2)+v(r_1-r_3)+v(r_2-r_3))$ where $r_1,r_2,r_3$ are the roots of $f$. Note that if the roots are integral, then all valuations are non-negative. Suppose the roots are not equidistant. Then the discriminant $\Delta_E$ is not a unit and this is true for all models of $E$ by i). Moreover, the cluster picture is unaffected by base change (only depths are affected) and hence $E$ must have potentially multiplicative reduction.

Now assume that the roots $r_1, r_2, r_3$ of $f$ are equidistant. We shall find a minimal Weierstrass model of $E$ over some finite extension $M$ which has unit discriminant. Observe that $v_M(\Delta_E)=6v_M(r_1-r_2)$ so $v_M(\Delta_E)=0$ if and only if $v_M(r_1-r_2)=0$, where $v_M$ is the normalised valuation of $M$.

Let $L \supset K(\cR)$ have ramification degree divisible by $2\denom v(r_1-r_2)$ over $K$ and note that $v_L(r_1-r_2)=2m \in 2\zz$. We claim that $E$ has good reduction over $L$.

Let $\pi_L$ be a uniformiser of $L$. Make the change of variables $X=\pi_L^{-2m}x, Y=\pi_L^{-3m}y$ to get an integral model for $E:Y^2=F(X)$ over $L$. Note the roots $r_1', r_2', r_3'$ of $F$ are still equidistant by i), and that $v_L(r_1'-r_2')=0$. Hence $E$ has unit discriminant and good reduction over $L$.
\end{proof}

We shall now study each of the two cluster pictures individually to infer more information.

\subsection{Potentially good}

We start with the case of potentially good reduction and hence have the cluster picture below. We have a single proper cluster $\cR$ with depth $d_{\cR}$.

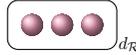
\begin{center}
\clusterpicture
\Root(0.00,2)B(a1);     
\Root(0.40,2)B(a2);     
\Root(0.80,2)B(a3);      
\ClusterL(c1){(a1)(a2)(a3)}{$d_{\cR}$};
\endclusterpicture
\end{center}

\begin{theorem}
\label{ecgood}
	Suppose $p \geqslant 5$ and let $E/K$ be an elliptic curve with potentially good reduction. Let $d_{\cR}$ be the depth of the single proper cluster. Then the Kodaira type, inertia representation $H^1_{\acute{e}t}(E/\overline{K},\qq_{\ell})$, $\ell \neq p$, and root number of $E/K$ is given in the following table.
	\begin{center}
	\begin{tabular}{c|c|c|c}
		$d_{\cR} \mod{2}$ & Kodaira type & $H^1_{\acute{e}t}(E/\overline{K},\qq_{\ell})$ & $W(E/K)$ \\ \hline 
		$0$ & $I_0$ & $2\rho_1$ & 1 \\
		$\frac{1}{3}$ & $II$ & $\rho_6$ & $\left(\frac{-1}{q}\right)$ \\
		$\frac{1}{2}$ & $III$ & $\rho_4$ & $\left(\frac{-2}{q}\right)$ \\
		$\frac{2}{3}$ & $IV$ & $\rho_3$ & $\left(\frac{-3}{q}\right)$ \\
		$1$ & $I_0^*$ & $2\rho_2$ & $\left(\frac{-1}{q}\right)$ \\
		$\frac{4}{3}$ & $IV^*$ & $\rho_3$ & $\left(\frac{-3}{q}\right)$ \\
		$\frac{3}{2}$ & $III^*$ & $\rho_4$ & $\left(\frac{-2}{q}\right)$ \\
		$\frac{5}{3}$ & $II^*$ & $\rho_6$ & $\left(\frac{-1}{q}\right)$
	\end{tabular}
	\end{center}
\end{theorem}

\begin{proof}
	Observe that by a change of variables over $K$, we may alter $d_{\cR}$ by an element of $2\zz$ and hence assume that assume that $0 \leqslant d_{\cR}<2$. Moreover, by a similar argument to that used in Lemma \ref{potred}, we may further assume that our model is a minimal Weierstrass model.

	With this assumption, we can now use the table in \cite[p.365]{Sil13} to recover the Kodaira type from the discriminant of the model. Note that the root number is computable from $H^1_{\acute{e}t}(E/\overline{K},\qq_{\ell})$ via Theorem \ref{rootthm}, so we only concentrate on computing $H^1_{\acute{e}t}(E/\overline{K},\qq_{\ell})$.
	
	By assumption on $p$, $E$ attains good reduction over a tamely ramified extension. Moreover, the ramification degree of a minimal extension completely determines $H^1_{\acute{e}t}(E/\overline{K},\qq_{\ell})$ as an inertia representation.

Applying a result of Rohrlich \cite[Proposition 2(v)]{Roh93}, the ramification degree is equal to $$\dfrac{12}{\gcd(12, v(\Delta_f))}=\dfrac{12}{\gcd(12,6d_{\cR})},$$ where $\Delta_f$ is the discriminant of $f$. Since $\denom d_{\cR} \leqslant 3$ by Theorem \ref{main}i, we can break into cases to see how $H^1$ depends on $d_{\cR}$; considering each case now yields the result.
\end{proof}

\begin{remark}
	One could also calculate $H^1_{\acute{e}t}(E/\overline{K},\qq_{\ell})$ by computing the ramification degree of $K(E[2],\Delta_f^{1/4})$ (cf. \cite[p.362]{Kra90}), together with Theorem \ref{main}i. The output in this case is $\lcm (\denom d_{\cR}, \denom \frac{3d_{\cR}}{2})$ which agrees with our computation.
\end{remark}

\subsection{Potentially multiplicative}

We now consider the potentially multiplicative case, where the cluster picture is as follows:

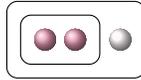
\begin{center}
\clusterpicture
\Root(0.00,2)B(a1);     
\Root(0.40,2)B(a2);     
\Root(1.00,2)A(a3);      
\ClusterL(c1){(a1)(a2)}{};
\ClusterL(c3){(c1)(a3)}{};
\endclusterpicture
\end{center}

We label the roots $r_1,r_2,r_3$ from left to right and define the two proper clusters as $\cR$, $\ss_1=\{r_1,r_2\}$.  Note that when $E$ has potentially multiplicative reduction, the representation $H^1_{\acute{e}t}(E/\overline{K},\qq_{\ell})$ is necessarily tamely ramified whenever $p \neq 2$ so we do not need to impose any further restrictions.

\begin{theorem}
\label{ecmult}
	Let $E/K: y^2=f(x)$ be an elliptic curve and suppose that $f$ is monic. Then $E/K$ has multiplicative reduction if and only if $d_{\cR} \in 2\zz$.
\end{theorem}

\begin{proof}
	Observe first that $d_{\cR} \in \zz$ by Theorem \ref{main}iii since both childern of $\cR$ are necessarily fixed. We now show that we may suppose that $f \in \mathcal{O}_K[x]$ and $d_{\cR} \in \{0,1\}$. Note that we may alter $d_{\cR}$ by an element of $2\zz$ via a substitution of the form $x_1=\pi^{-2k}x, y_1=\pi^{-3k}y$, $k \in \zz$, and moreover this is the only method of changing the depths. Hence we may suppose that $d_{\cR} \in \{0,1\}$. Consider $r_3$, the unique root not in $\ss_1$. This is necessarily Galois-invariant and hence $r_3 \in K$; the model $y^2=f(x+r_3)$ is the required integral model for $E/K$ and furthermore is a minimal Weierstrass model. 
	
	By Lemma \ref{potred}ii, $E/K$ has potentially multiplicative reduction. Moreover, it is multiplicative if and only if the reduction of $f$ has two distinct roots in the (algebraic closure) of the residue field of $K$; this is true if and only if $d_{\cR}=0$.
\end{proof}

\begin{remark}
\label{kodmult}
	Since the depths only ``see'' the inertia action, we are unable to distinguish between split and nonsplit multiplicative reduction. Indeed, in \cite{DDMM}, the authors introduce extra labelling on the cluster picture to record the Frobenius action which would then suffice to differentiate between these two reduction types. We may however recover the Kodaira type again using the valuation of the discriminant.
\end{remark}

Before we close this section, we briefly note that we have proved the validity of the inertia formula of Theorem \ref{M2D2rep} for elliptic curves.

\begin{theorem}
\label{kodthm}
	Suppose $p \geqslant 5$. Then the inertia representation formula given in Theorem \ref{M2D2rep} for $H^1_{\acute{e}t}(E/\overline{K},\qq_{\ell})$, $\ell \neq p$, also holds for elliptic curves $E/K$. Moreover, the depths of the associated cluster picture uniquely determine the Kodaira type of $E/K$.
\end{theorem}

\begin{proof}
	This follows by a simple case-by-case check of the formula against the results of Theorems \ref{ecgood} and \ref{ecmult}, as well as Remark \ref{kodmult} for the Kodaira types.
\end{proof}

\section{An extended example}
\label{egsection}

Throughout this section, we will apply our various results to the following cluster picture. 

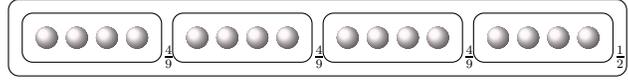
\begin{center}
\clusterpicture
\Root(0.00,2)A(a1);     
\Root(0.40,2)A(a2);     
\Root(0.80,2)A(a3);     
\Root(1.20,2)A(a4);
\Root(2.00,2)A(b1);
\Root(2.40,2)A(b2);
\Root(2.80,2)A(b3);
\Root(3.20,2)A(b4);
\Root(4.00,2)A(c1);
\Root(4.40,2)A(c2);
\Root(4.80,2)A(c3);
\Root(5.20,2)A(c4);
\Root(6.00,2)A(d1);
\Root(6.40,2)A(d2);
\Root(6.80,2)A(d3);
\Root(7.20,2)A(d4);
\ClusterL(s1){(a1)(a2)(a3)(a4)}{$\frac{4}{9}$};
\ClusterL(s2){(b1)(b2)(b3)(b4)}{$\frac{4}{9}$};
\ClusterL(s3){(c1)(c2)(c3)(c4)}{$\frac{4}{9}$};
\ClusterL(s4){(d1)(d2)(d3)(d4)}{$\frac{1}{2}$};
\ClusterL(s5){(s1)(s2)(s3)(s4)}{$\frac{1}{3}$};
\endclusterpicture
\end{center}

We order the roots from left to right as $r_1,\cdots, r_{16}$. We label the proper clusters as 
\begin{eqnarray*}
	\cR, \qquad \ss_1=\{r_1,r_2,r_3,r_4\}, \qquad \ss_2&=&\{r_5,r_6,r_7,r_8 \}, \\
	\ss_3=\{r_9,r_{10},r_{11},r_{12} \}, \qquad \ss_4&=&\{r_{13},r_{14},r_{15},r_{16}\},
\end{eqnarray*}
where the number in the bottom left corner of the cluster denotes its depth.

We split the example into three subsections to demonstrate our three main results that we can apply:
\begin{itemize}
	\item Determine the inertia action and identify orphans; 
	\item Construct a polynomial that gives rise to the cluster picture;
	\item Compute the inertia representation of the corresponding hyperelliptic curve.
\end{itemize} 

\subsection{Inertia action and orphans}
We begin the analysis of our example by determining the inertia action on all clusters with their stabilisers and the identification of orphans to demonstrate the application of Theorem \ref{main}. A priori, this does assume that we know the cluster picture is already of polynomial type, but the observant reader will notice that this is equivalent to checking that the cluster picture satisfies Hypothesis H. Hence the assumption is moot since if we arrive at a contradiction then it would not be of polynomial type for $p \geqslant 17$.

Firstly, we observe that $|I|=\lcm(9,2,3)=18$. We now proceed with determining the inertia action. Note that $\denom d_{\cR}=3$ so the non-orphans of $\cR$ have orbit length $3$, however $\cR$ has four children so it must have an orphan. Since valuations are Galois-invariant, the orphan is necessarily $\ss_4$. Without loss of generality, we assume that a generator of inertia $i$ acts as the cycle $(\ss_1,\ss_2,\ss_3)$.

Now consider $\ss_1$. This has index $[I:I_{\ss_1}]=3$ so the orbit length of its non-orphans is $\denom \frac{3 \times 4}{9}=3$, hence there is once again a orphan. Since all children are isomorphic\footnote{We say two clusters are isomorphic if there is a depth-preserving bijection between them and all subclusters.}, it does not matter which child we choose to be the orphan and we therefore select $r_4$. Without loss of generality, we hence suppose that $i^3$ acts as $(r_1,r_2,r_3)$. As $\ss_2$ and $\ss_3$ are inertia-conjugate to $\ss_1$ they also have an orphan, which we again assume to be the rightmost root, with $i^3$ otherwise acting left to right as before.

Finally we study the action on $\ss_4$. Since this was an orphan, we have $[I:I_{\ss_4}]=1$ and hence the orbit length on its children is $\denom \frac{1}{2}=2$. As it has four children, there are no orphans and therefore we may assume that $i$ acts as $(r_{13},r_{14})(r_{15},r_{16})$.

The complete inertia action on roots is therefore given as $$I \cong \langle (r_1,r_5,r_9,r_2,r_6,r_{10},r_3,r_7,r_{11})(r_4,r_8,r_{12})(r_{13},r_{14})(r_{15},r_{16}) \rangle.$$

To encapsulate the extra information of the orphans, we colour them green to produce what we shall henceforth refer to as an \emph{orphan cluster picture}. Our example as an orphan cluster picture now appears below.

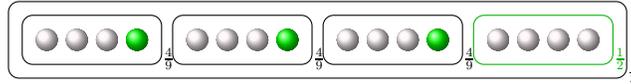
\begin{center}
\clusterpicture
\Root(0.00,2)A(a1);     
\Root(0.40,2)A(a2);     
\Root(0.80,2)A(a3);      
\Root(1.20,2)G(a4);
\Root(2.00,2)A(b1);
\Root(2.40,2)A(b2);
\Root(2.80,2)A(b3);
\Root(3.20,2)G(b4);
\Root(4.00,2)A(c1);
\Root(4.40,2)A(c2);
\Root(4.80,2)A(c3);
\Root(5.20,2)G(c4);
\Root(6.00,2)A(d1);
\Root(6.40,2)A(d2);
\Root(6.80,2)A(d3);
\Root(7.20,2)A(d4);
\ClusterL(s1){(a1)(a2)(a3)(a4)}{$\frac{4}{9}$};
\ClusterL(s2){(b1)(b2)(b3)(b4)}{$\frac{4}{9}$};
\ClusterL(s3){(c1)(c2)(c3)(c4)}{$\frac{4}{9}$};
\ClusterG(s4){(d1)(d2)(d3)(d4)}{$\frac{1}{2}$};
\ClusterL(s5){(s1)(s2)(s3)(s4)}{$\frac{1}{3}$};
\endclusterpicture
\end{center}

With this extra information, we shall verify Theorem \ref{main}iv for three of the roots.
\begin{eqnarray*}
	[I:I_{r_1}] &=& \lcm(\denom \ss_1,\denom \cR ) = 9; \\ \relax
	[I:I_{r_4}] &=& \lcm(1,\denom \cR ) = 3; \\ \relax
	[I:I_{r_{13}}] &=& \lcm(\denom \ss_4,1 ) = 2.
\end{eqnarray*}

This check is sufficient to show that the cluster picture does indeed satisfy Hypothesis H and is therefore of polynomial type.

\subsection{Construction of the polynomial}
Now that we know our cluster picture is of polynomial type, we shall associate a polynomial to it via Construction \ref{polycons}. We begin by choosing a compatible set $Y$ of representatives for roots modulo the inertia action; this surmounts to choosing any pair of non-conjugate roots of the orphan cluster $\ss_4$ and otherwise selecting one of $\ss_1,\ss_2,\ss_3$ and using both the orphan root and any non-orphan root of that cluster. Recalling our left-to-right ordering of the roots, we choose $$Y=\{r_1,r_4,r_{13},r_{15}\}.$$

This leads us to finding a compatible set of coefficients $a(y,\ss)$ for the following:
\begin{eqnarray*}
	\alpha(r_1) &=& a(r_1,\ss_1)\pi_K^{4/9} + a(r_1,\cR)\pi_K^{1/3}, \\
	\alpha(r_4) &=& a(r_4,\ss_1)\pi_K^{4/9} + a(r_4,\cR)\pi_K^{1/3}, \\
	\alpha(r_{13}) &=& a(r_{13},\ss_4)\pi_K^{1/2} + a(r_{13},\cR)\pi_K^{1/3}, \\
	\alpha(r_{15}) &=& a(r_{15},\ss_4)\pi_K^{1/2} + a(r_{15},\cR)\pi_K^{1/3}. \\
\end{eqnarray*}

First we note that Construction \ref{polycons}(ii)(a) implies $a(r_1,\cR)=a(r_4,\cR)$ and $a(r_{13},\cR)=a(r_{15},\cR)$. Moreover, $a(r_4,\ss_1)=a(r_{13},\cR)=0$ by Construction \ref{polycons}(iii). All that remains is to make the remaining terms sufficiently different; indeed we find that the following choices work with the product $f_Y$ of their minimal polynomials giving an isomorphic cluster picture. 

\begin{eqnarray*}
	\alpha(r_1) &=& \pi_K^{4/9} + \pi_K^{1/3}, \\
	\alpha(r_4) &=& \pi_K^{1/3}, \\
	\alpha(r_{13}) &=& \pi_K^{1/2}, \\
	\alpha(r_{15}) &=& 2\pi_K^{1/2}. \\
\end{eqnarray*}

The original bound on the primes we gave stipulated this works for $p>16$, however we can see that these choices are still valid for $p>3$ (cf. Remark \ref{primecon}). If $K=\qq_{19}, \pi_K=19$, then this yields $f_Y$ as the polynomial $$(x^2-19)(x^2-76)(x^3-19)(x^9-57x^6-6498x^4+1083x^3-61731x^2-61731x-137180).$$

\subsection{Inertia representation of the curve}

Recall the orphan cluster picture we used earlier, given below. We shall compute the corresponding inertia representation on $H^1(C)=H^1_{\acute{e}t}(C/\overline{K},\qq_{\ell})$ of the associated hyperelliptic curve $C$.

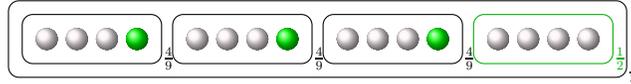
\begin{center}
\clusterpicture
\Root(0.00,2)A(a1);     
\Root(0.40,2)A(a2);     
\Root(0.80,2)A(a3);      
\Root(1.20,2)G(a4);
\Root(2.00,2)A(b1);
\Root(2.40,2)A(b2);
\Root(2.80,2)A(b3);
\Root(3.20,2)G(b4);
\Root(4.00,2)A(c1);
\Root(4.40,2)A(c2);
\Root(4.80,2)A(c3);
\Root(5.20,2)G(c4);
\Root(6.00,2)A(d1);
\Root(6.40,2)A(d2);
\Root(6.80,2)A(d3);
\Root(7.20,2)A(d4);
\ClusterL(s1){(a1)(a2)(a3)(a4)}{$\frac{4}{9}$};
\ClusterL(s2){(b1)(b2)(b3)(b4)}{$\frac{4}{9}$};
\ClusterL(s3){(c1)(c2)(c3)(c4)}{$\frac{4}{9}$};
\ClusterG(s4){(d1)(d2)(d3)(d4)}{$\frac{1}{2}$};
\ClusterL(s5){(s1)(s2)(s3)(s4)}{$\frac{1}{3}$};
\endclusterpicture
\end{center}

Observe first that $X/I=\{ \ss_1, \ss_4 \}$ since $\cR$ is \"{u}bereven. We first tabulate a few of the invariants associated to each of these clusters, which one can extract through the orphan cluster picture as we have previously done.

\begin{center}
\begin{tabular}{c|ccccccc}
$\ss$ & $n_{\ss}$ & $n'_{\ss}$ & $|\ss^{odd}|$ & $\left\lfloor \dfrac{|\ss^{odd}|}{n'_{\ss}} \right\rfloor$ & Orphan? & $\ord \gamma_{\ss}$ & $\ord \es$ \\ \hline
$\ss_1$ & $3$ & $3$ & $4$ & $1$ & Yes & $3$ & $1$ \\
$\ss_4$ & $1$ & $2$ & $4$ & $2$ & No & $1$ & $1$
\end{tabular}
\end{center}

We begin with $\ss_1$. Now since $\varepsilon_{\ss_1}=\mathds{1}$, we have $\Ind_{I_{\ss_1}}^I \varepsilon_{\ss_1} =\bigoplus\limits_{m|3} \rho_m = \rho_1 \oplus \rho_3$. Moreover, since $\ss_1$ has an odd orphan, there is no contribution from the middle line of the formula in Theorem \ref{repthm}. Lastly, since $n'_{\ss}=3$, we need to consider the sets $S_{1,3}=\{3\}$ and $S_{3,3}=\{1,3\}$; for transparency of the computation we display the information in a table. Since the divisor $n_1$ only affects the order and not the multiplicity of the representation, we have incorporated all possible choices into the table as well.

\begin{center}
\begin{tabular}{c|ccccc}
	$(d,t,s)$ & $\gcd(n_{\ss},s^{\infty})$ & $\alpha_{d,t,s}$ & $\beta(n_{\ss},s)$ & $n_1$ & $s\gcd(n_{\ss},s^{\infty})$ \\ \hline
	$(1,3,3)$ & $3$ & $\frac{1}{2}$ & $1$ & $1$ & $9$ \\
	$(3,3,1)$ & $1$ & $1$ & $1$ & $1,3$ & $1,3$ \\
	$(3,3,3)$ & $3$ & $\frac{1}{2}$ & $1$ & $1$ & $9$
\end{tabular}
\end{center}

Hence $$\Ind V_{\ss_1} = \frac{1}{2}\rho_9 \oplus \rho_1 \oplus \rho_3 \oplus \frac{1}{2}\rho_9 \ominus (\rho_1 \oplus \rho_3)=\rho_9.$$

We now proceed to do a similar computation for $\ss_4$. In this case $\Ind_{I_{\ss_4}}^I \varepsilon_{\ss_4}=\rho_1$ and we have a contribution from the middle line as there is no orphan; however it is straightforward to compute that this is simply $\ominus \rho_1$ since $n_{\ss_4}=\ord \gamma_{\ss_4}=1$. We now reproduce the table above for this cluster where the sets of interest are $S_{1,1}=\{1\}$ and $S_{2,1}=\{2\}$.

\begin{center}
\begin{tabular}{c|ccccc}
	$(d,t,s)$ & $\gcd(n_{\ss},s^{\infty})$ & $\alpha_{d,t,s}$ & $\beta(n_{\ss},s)$ & $n_1$ & $s\gcd(n_{\ss},s^{\infty})$ \\ \hline
	$(1,1,1)$ & $1$ & $1$ & $1$ & $1$ & $1$ \\
	$(2,1,2)$ & $1$ & $1$ & $1$ & $1$ & $2$
\end{tabular}
\end{center}

Recalling that $\left\lfloor \dfrac{|\ss_4^{odd}|}{n'_{\ss_4}} \right\rfloor =2$, we hence have $$\Ind V_{\ss_4} = 2(\rho_1 \oplus \rho_2) \ominus \rho_1 \ominus \rho_1 = 2\rho_2.$$

Lastly, using our previous computations for $U_{\ss}$ along with the fact that $\varepsilon_{\cR}=\mathds{1}$ we finally have that $$H^1(C) = \underbrace{2\rho_2 \oplus \rho_9}_{H^1_{ab}} \oplus ((\underbrace{\rho_1 \oplus \rho_3}_{H^1_t}) \otimes \sp (2)).$$

\appendix
\section{Proof of Theorem \ref{main}}

In this section, we shall prove Theorem \ref{main} through a series of lemmas, from which the theorem will result.\footnote{When life hands you lemmas, make lemma-nade.} We briefly restate the theorem below and set-up a little bit more notation that we shall use throughout this section for the proof.

\begin{theorem}
\label{mainpf}
	Assume the inertia group $I$ is tame.
	\begin{enumerate}
		\item[]
    	\item $\lcm_{\ss} \denom \ds = |I|$, where the $\lcm$ runs over all proper clusters.
    	\item Fix a cluster $\ss$. Then the orbits of non-orphans of $\ss$ under $\Is$ all have equal length.
    	\item Let $\ss'$ be a child of $\ss$ which is not an orphan. Then the length of its orbit under $\Is$ is $\denom(\ds [I:\Is])$.
    	\item Let $\ss$ be a cluster. Then $$[I:\Is] = \lcm_{\ss \subsetneq \ss'} \denom d_{\ss'}^*,$$ where for a cluster $\ss' \supsetneq \ss$, $$d_{\ss'}^*=\begin{cases}
		1 &\text{ if the child of $\ss'$ containing $\ss$ is an orphan}, \\
		d_{\ss'} &\text{ else}.
	\end{cases}$$
	\end{enumerate}
\end{theorem}

\noindent \textbf{Notation.} \\
\begin{tabular}{cl}
$v_q$ & $q$-adic valuation for a prime $q$; \\
$i$ & generator of the (cyclic) inertia group $I$; \\
$e$ & $=|I|$; \\
$\varpi$ & a uniformiser of $K(\cR)$.
\end{tabular}

\begin{lemma}
\label{thm1}
	Let $r$ be a root and suppose it has an orbit of length $n$ under $I$. Then $$n =\lcm_{1 \leqslant k \leqslant n-1} \denom v(r-i^k(r)).$$ 
\end{lemma}

\begin{proof}
	As the orbit length of $r$ is $n$, the splitting field $F \subset K(\cR)$ of the minimal polynomial of $r$ has ramification degree $n$. Write $r=\sum\limits_{j \geqslant 0} a_j\pi_F^j$ where the nonzero $a_j$ have valuation $0$, $\pi_F$ is a uniformiser of $F$ and note that $v(\pi_F)=\frac{1}{n}$. Then $r-i^k(r)=\sum\limits_{j \geqslant 0} a_j(1-\zeta^{kj})\pi_F^j$ for some primitive $n^{th}$ root of unity $\zeta$.
	
	Suppose that $n=q^m$ is a prime power. Then any nonzero term of $r-i^{q^{m-1}}(r)$ has valuation $\frac{a}{q^m}$ for some positive integer $a$ coprime to $q$ so the statement holds. Considering the prime factorisation yields the general case.
	
	
\end{proof}

\begin{remark}
Note that one cannot drop the $\lcm$ from the preceding lemma. For example, $r=p^{1/2}+p^{2/3}$ has orbit $6$ under the absolute inertia group of $\qp$ but $v(r-i(r))$ never has denominator $6$ for any inertia element $i$. This is also clear from the cluster picture of the minimal polynomial of $r$ given below.

\begin{center}
\clusterpicture
\Root(0.00,2)A(a1);
\Root(0.40,2)A(a2);
\Root(0.80,2)A(a3);
\Root(1.80,2)A(a4);
\Root(2.20,2)A(a5);
\Root(2.60,2)A(a6)
\ClusterL(s1){(a1)(a2)(a3)}{$\frac{2}{3}$};
\ClusterL(s2){(a4)(a5)(a6)}{$\frac{2}{3}$};
\ClusterL(s5){(s1)(s2)}{$\frac{1}{2}$};
\endclusterpicture
\end{center}
\end{remark}

\begin{lemma}[=Theorem \ref{mainpf}i]
	$\lcm_{\ss} \denom \ds = e$, where the $\lcm$ runs over all proper clusters.
\end{lemma}

\begin{proof}
	Fix a prime $q$ and let $v_q(e)=a$. Since $e$ is the order of the inertia subgroup of $\Gal(K(\cR)/K)$, there exists $r \in \cR$ such that the order of its inertia orbit is divisible by $q^a$. Hence by Lemma \ref{thm1}, there exists an integer $k_r$ such that $v_q(\denom v(r-i^{k_r}(r)))=a$. Let $\ss$ be the smallest cluster containing $r$ and $i^{k_r}(r)$. Then $q^a| \denom \ds$. Iterating this for each prime divisor of $e$, we see that $e| \lcm_{\ss} \denom \ds$.
	
	Conversely, suppose $\denom \ds= \denom v(r-r')=m$ for some $r,r' \in \ss$. Since $r-r' \in K(\cR)$, we necessarily have $m|e$ and the result follows.
\end{proof}

\begin{remark}
	This gives a necessary condition for the extension $K(\cR)/K$ to be tamely ramified, namely that $p \nmid \denom \ds$ for all proper clusters $\ss$. This is not sufficient however, since the polynomial $x^3+3x+3 \in \zz_3[x]$ has a (wild) $S_3$ inertia action, but the depth of its only proper cluster is $\frac{1}{2}$.
\end{remark}

\begin{remark}
\label{Galval}
	Note that as valuations are Galois invariant, the depths of $\ss$ and $i(\ss)$ are equal for all clusters $\ss$ and $i \in I$.
\end{remark}

\begin{lemma}[=Theorem \ref{mainpf}ii]
	Let $\ss$ be a cluster. Then the orbits of the children of $\ss$ under $\Is$ all have equal length, except for possibly one fixed child.
\end{lemma}

\begin{proof}
We will use the following claim:

\begin{claim}
	Let $\alpha, \beta \in \cR$. Let $\ss$ be a cluster of depth $\ds=\frac{N}{e}$ and suppose that $\alpha \in \ss$. Then $\beta \in \ss$ if and only if $\alpha \equiv \beta \bmod{\varpi^N}$. 
	
	Moreover, if $\beta \in \ss$, then $\alpha$ and $\beta$ are in the same child of $\ss$ if and only if $\alpha \equiv \beta \bmod{\varpi^{N+1}}$.
\end{claim}

	Let $\ss'$ be the child containing $\alpha=\sum\limits_{j \geqslant 0}a_j\varpi^j$ where the nonzero $a_j$ ahve valuation $0$. Write $\Is=\langle \sigma \rangle$ and note that $\sigma(\alpha)=\sum\limits_{j=0}^{N-1} a_j \varpi^j + \sum\limits_{j \geqslant N} a_j \sigma(\varpi^j)$ since $\sigma \in \Is$.
	
	If $a_N=0$, then $a_N\varpi^N=a_N\sigma(\varpi^N)$ so $\ss'$ is fixed by $\Is$ by the second part of the claim. Otherwise $a_N \neq 0$ and the size of the orbit of $\ss'$ is equal to the size of the orbit of $\sigma$ on $\varpi^N$; this is independent of $a_N$ and hence the orbits of all such children will have equal length.
	
	All that remains is to prove the claim. Since $v(\varpi)=\frac{1}{e}$, we have that
$$\beta \in \ss \Leftrightarrow v(\alpha-\beta) \geqslant \ds=\frac{N}{e} \Leftrightarrow\sum\limits_{j=0}^{N-1} a_j \varpi^j = \sum\limits_{j=0}^{N-1} b_j \varpi^j.$$ 
	
	Similarly, $$\beta \in \ss' \Leftrightarrow v(\alpha-\beta)>\ds \Leftrightarrow \sum\limits_{j=0}^N a_j \varpi^j = \sum\limits_{j=0}^N b_j \varpi^j.$$
\end{proof}

\begin{remark}
	The existence of a child with $a_N=0$ is precisely what we refer to as a \emph{orphan}, which proves its uniqueness if it exists.
\end{remark}

\begin{lemma}[=Theorem \ref{mainpf}iii]
\label{staborder}
	Let $\ss'$ be a non-orphan of $\ss$. Then the length of the orbit of $\ss'$ under $\Is$ is equal to $\denom (\ds [I:\Is])$.
\end{lemma}

\begin{proof}
	First let $\ds=\frac{N}{e}$, $\alpha = \sum\limits_{j \geqslant 0} a_j\varpi^j \in \ss'$ and $[I:\Is]=k$ so $\Is=\langle i^k \rangle$. As $\ss'$ is not an orphan, $a_N \neq 0$. Then
	\begin{eqnarray*}
		(i^k)^l(\alpha) \in \ss' &\Leftrightarrow & i^{kl}(\varpi^N)=\varpi^N, \\
								  &\Leftrightarrow & e|klN  \qquad \qquad \text{ as }i(\varpi)=\zeta_e\varpi, \\
								  &\Leftrightarrow & \frac{e}{\gcd(e,kN)} \, | l,
	\end{eqnarray*}
	and hence the orbit length is precisely $\frac{e}{\gcd(e,kN)}$. Now let $\ds=\frac{a}{b}=\frac{N}{e}$ with $\gcd(a,b)=1$ and therefore $N=a\gcd(N,e), e=b\gcd(N,e)$. Hence
	
	\begin{eqnarray*}
	  \frac{e}{\gcd(e,kN)} &=& \frac{b \gcd(N,e)}{\gcd(b\gcd(N,e),ak\gcd(N,e))}, \\
	  &=& \frac{b}{\gcd(b,ak)}, \\
	  &=& \frac{b}{\gcd(b,k)},
	\end{eqnarray*}
	since $\gcd(a,b)=1$. As $\denom(\ds [I:\Is])=\frac{b}{\gcd(b,k)}$, we are done.
\end{proof}

\begin{corollary}
\label{lcmcor}
	Let $\ss$ be a cluster and $\ss'$ a non-orphan of $\ss$. Then the orbit of $\ss'$ under $I$ is equal to $\lcm (\denom \ds, [I:\Is])$.
\end{corollary}

\begin{proof}
	From Lemma \ref{staborder}, the length under $I$ is $$[I:\Is] \times \denom(\ds[I:\Is])= [I:\Is]\dfrac{\denom \ds}{\gcd(\denom \ds, [I:\Is])}.$$
\end{proof}

\begin{lemma}[=Theorem \ref{mainpf}iv]
\label{mainiv}
	Let $\ss$ be a cluster. Then $$[I:\Is] = \lcm_{\ss \subsetneq \ss'} \denom d_{\ss'}^*,$$ where for a cluster $\ss' \supsetneq \ss$, $$d_{\ss'}^*=\begin{cases}
		1 &\text{ if the child of $\ss'$ containing $\ss$ is an orphan}, \\
		d_{\ss'} &\text{ else}.
	\end{cases}$$
\end{lemma}

\begin{proof}
	First suppose that neither $\ss$ nor any ancestor of $\ss$ is an orphan. If $\ss$ is the top cluster then the result follows trivially. Otherwise, by Corollary \ref{lcmcor}, the orbit length of $\ss$ (equivalently the index of the stabiliser) under $I$ is $\lcm ([I:I_{P(\ss)}], \denom d_{P(\ss)}).$ By induction, we hence have $$[I:\Is]=\lcm_{\ss \subsetneq \ss'} \denom d_{\ss'}.$$
	
	Now suppose that $\ss$ is an orphan, but no ancestor\footnote{Continuing the familial nomenclature, an ancestor of $\ss$ is a cluster properly containing $\ss$.} of $\ss$ is. Then $$[I:\Is]=[I:I_{P(\ss)}]=\lcm _{P(\ss) \subsetneq \ss'} \denom d_{\ss'} = \lcm _{\ss \subsetneq \ss'} \denom d_{\ss'}^*.$$
	
	More generally, we see that if some ancestor $\ss'$ is an orphan, then $[I:I_{\ss'}]=[I:I_{P(\ss')}]$ and so we omit the term $d_{P(\ss')}$ and iterate inwards.
\end{proof}

\begin{remark}
	In \cite{DDMM}, the authors also introduce the notion of relative depth of a (proper) cluster, defined by $$\delta_{\ss}=\begin{cases} d_{\ss}-d_{P(\ss)} &\text{ if } \ss \neq \cR, \\
	d_{\cR} &\text{ if } \ss=\cR.
	\end{cases}$$
	One can then ask whether the results of Theorem \ref{mainpf} hold when the depth is replaced by the relative depth instead. In fact, it holds that $\lcm_{\ss} \denom \delta_{\ss} = |I|$ since for two distinct rational numbers $\alpha_1,\alpha_2$, we have the equality $$\lcm (\denom \alpha_1, \denom \alpha_2) = \lcm (\denom \alpha_1, \denom (\alpha_1-\alpha_2)).$$ Unfortunately, the analogous version of Theorem \ref{mainpf}iii (and hence also iv) fail; a counterexample is given by the polynomial $$f=(x^2-p)(x^3-p^2) \in \qp[x],$$ whose orphan cluster picture (with standard depths) is below.
	
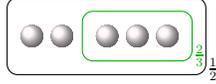
\begin{center}
\clusterpicture
\Root(0.00,2)A(a1);     
\Root(0.40,2)A(a2);     
\Root(1.00,2)A(d1);
\Root(1.40,2)A(d2);
\Root(1.80,2)A(d3);
\ClusterG(s4){(d1)(d2)(d3)}{$\frac{2}{3}$};
\ClusterL(s5){(a1)(a2)(s4)}{$\frac{1}{2}$};
\endclusterpicture
\end{center}

Now the orphan cluster $\ss$ trivially has stabiliser $I$, and all roots in $\ss$ have order $3$ under $I_{\ss}$. However, $\delta_{\ss}=\frac{1}{6}$ which proves our counterexample as $3 \neq \denom(\delta_{\ss}[I:\Is])$.
\end{remark}

\section{Proof of Theorem \ref{repthm}}
In this section, we shall prove the closed formula for the inertia representation. In particular, given a cluster $\ss$, we are interested in a closed formula for the induced representation $\Ind_{\Is}^I V_{\ss}$, where $$V_{\ss} = \gamma_{\ss} \otimes (\cc[\ss^{odd}] \ominus \mathds{1}) \ominus \epsilon_{\ss}.$$

We restate the main theorem below to begin with, before proving a couple of representation theoretic lemmas and circling back to the theorem.

\begin{theorem}
\label{repthmpf}
	Let $C/K: y^2=f(x)$ be a hyperelliptic curve with tame reduction. Let $\ss$ be a cluster and suppose that $\gamma_{\ss}$ has order $t$. Then 
	\begin{eqnarray*}
		\Ind_{\Is}^I V_{\ss} &=&  \left\lfloor \dfrac{|\ss^{odd}|}{n'_{\ss}} \right\rfloor \bigoplus\limits_{d|n'_{\ss}} \bigoplus\limits_{s \in S_{d,t}} \bigoplus\limits_{n_1|\frac{n_{\ss}}{\gcd(n_{\ss},s^{\infty})}} \alpha_{d,t,s} \beta(n_{\ss},s) \rho_{s\gcd(n_{\ss},s^{\infty})n_1} \\
		&\oplus& \left( |\ss^{odd}| - n'_{\ss}\left\lfloor \dfrac{|\ss^{odd}|}{n'_{\ss}} \right\rfloor -1 \right) \frac{\beta(n_{\ss},t)}{\varphi(t)} \bigoplus\limits_{n_2|\frac{n_{\ss}}{\gcd(n_{\ss},t^{\infty})}} \rho_{tn_2\gcd(n_{\ss},t^{\infty})} \\
		&\ominus & \Ind_{\Is}^I \es, 
	\end{eqnarray*}	
	where 
	
	$$\Ind_{\Is}^I \es = \begin{cases}
	0 &\text{if $\es=0$}; \\
	\bigoplus\limits_{m|n_{\ss}} \rho_m &\text{ if $\es$ is trivial}; \\
	\bigoplus\limits_{\substack{m|2n_{\ss} \\ m\nmid n_{\ss}}} \rho_m &\text{ if $\es$ has order $2$}.
	\end{cases}$$
\end{theorem}

To assist the reader in navigating this notational minefield, we will reintroduce the extra notation as required. First, recall that for an integer $d$, we let $\rho_d$ denote the direct sum of all characters of order $d$ of $\cc[C_d]$. We will begin by first describing the orders of the tensor product of characters, before examining how the induction of a character across cyclic groups breaks up.

\begin{lemma}
\label{twist}
Let $A$ be a non-trivial finite cyclic group and let $\gamma_n, \gamma'_m$ be characters of $A$ of prime power orders $n$ and $m$ respectively.
\begin{enumerate}
	\item If $n \neq m$, then $\gamma_n \otimes \gamma'_m$ has order $\lcm(n,m)$;
	\item If $n=m=q^r$ for some prime $q$, then $\gamma_n \otimes \gamma'_m$ has order dividing $n$. More generally, $$\gamma_n \otimes \rho_n = \frac{q-2}{q-1}\rho_n \oplus \bigoplus\limits_{j=0}^{r-1} \rho_{q^j}.$$
	
	Equivalently, $\gamma_n \otimes \rho_n$ consists exactly of $\varphi(q^k)$ characters of order $q^k$ for each $0 \leqslant k \leqslant r-1$ and $(q-2)q^{r-1}$ characters of order $q^r$, with every character summand distinct.
	\end{enumerate}
\end{lemma}

\begin{proof}
i): If $n,m$ are coprime, then $\gamma_n \otimes \gamma_m$ is necessarily a primitive character of $C_n \times C_m \cong C_{nm}$. Otherwise assume they are a power of the same prime and $n>m$ without loss of generality. Then $\gamma_n \otimes \gamma_m$ also has order $n$ since otherwise the order of $\gamma_n = \gamma_m^{-1} \otimes (\gamma_m \otimes \gamma_n)$ would be strictly smaller. Collating these two cases, we find the order is $\lcm(m,n)$. \\
ii): Identify the summands of $\rho_n$ with elements of $(\zz/q^r\zz)^{\times}$ and suppose without loss of generality that $\gamma_n=1$ under this identification. The tensor product then corresponds to addition so we are left to count the orders (under addition) of $a+1 \in \zz/q^r\zz$ when $a$ is a unit; this gives the claimed formula.
\end{proof}

We now want to describe $\gamma_t \otimes \rho_d$ in general. The second part of the above lemma is the reason for introducing the sets $A_{d,t}$ and $S_{d,t}$ to keep track of these cases. \\

\noindent \textbf{Notation.}

\begin{tabular}{ll}
$A_{d,t}$ & $= \{ q \text{ prime} \, | \, v_q(d)=v_q(t)>0 \}$; \\
$S_{d,t}$ &  $=\left\lbrace \left. \dfrac{\lcm(d,t)}{\prod\limits_{q_i \in A_{d,t}} q_i^{m_i}} \, \right| \, 0 \leqslant m_i \leqslant v_{q_i}(t) \right\rbrace;$ \\
$Q_{q,k}(s)$ & $=\begin{cases}
	\frac{q-2}{q-1} &\text{ if } v_q(s)=k \qquad \text{ where $q$ is prime}, \\
	1 &\text{ else }. \end{cases}$ \\
\end{tabular}
\vspace{10pt}

If $s \in S_{d,t}$, then we further define

\begin{tabular}{ll}
$\alpha_{d,t,s}$ & $=\dfrac{\varphi(d)}{\varphi(\lcm(d,t))}\prod\limits_{q_i \in A_{d,t}} Q_{q_i,v_{q_i}(t)}(s).$
\end{tabular}

\begin{lemma}
\label{gentwist}
Let $\rho_d$ be the direct sum of all characters of order $d$ of $\cc[C_d]$ and let $\gamma_t$ be a character of order $t$. Then $$\gamma_t \otimes \rho_d = \bigoplus\limits_{s \in S_{d,t}} \alpha_{d,t,s} \rho_s.$$
\end{lemma}

\begin{proof}
	Write $\gamma_t=\gamma_{t/t_A} \otimes \gamma_{t_A}$, where $t_A$ is the maximal divisor of $t$ not divisible by any prime in $A_{d,t}$. Then by Lemma \ref{twist}i, $\gamma_{t_A} \otimes \rho_d =\frac{\varphi(d)}{\varphi(l)} \rho_{l}$, where $l=\lcm(d,t_A)=\lcm(d,t)$. Now $\gamma_{t/t_A} = \bigotimes\limits_{q_i \in A_{d,t}} \gamma_{q_i^{a_i}}$, where $a_i=v_{q_i}(t)$.
	Using Lemma \ref{twist}ii, we observe that each time we twist by a $\gamma_{q_i^{a_i}}$, the order of the character is divided by all possible powers of $q_i$, and hence the list of orders of the summands of $\gamma_t \otimes \rho_d$ is precisely the set $S_{d,t}$. Moreover the multiplicities of each component representation $\rho_s$ is equal to $\alpha_{d,t,s}$ noting that it is equal to $\frac{\varphi(d)}{\varphi(l)}$, unless $v_{q_i}(s)=a_i$ is maximal.
\end{proof}

We now look at how the order of a character changes under induction. For integers $n,d>0$, we define the greatest common divisor of $n$ and $d^{\infty}$ to be $\gcd(n,d^{\infty}):= \lim\limits_{r \rightarrow \infty} \gcd(n,d^r)$.

\begin{lemma}
\label{ind}
	Let $\chi$ be a character of order $d$, $n>0$ a positive integer. Let $\tau= \Ind_{C_d}^{C_{dn}} \chi$. Let $g=\gcd(n,d^{\infty})$. Then $\tau$ consists of precisely $g\varphi(n_1)$ characters of order $dgn_1$ for each positive divisor $n_1$ of $\frac{n}{g}$.
\end{lemma}

\begin{proof}
We do this by first inducing $\chi$ to $C_{dn'}$ where $n'=\frac{n}{g}$. Now $\Ind_{C_d}^{C_{dn'}} \chi = \chi \otimes \cc[C_{n'}]$, which consists of $\varphi(n_1)$ characters of order $dn_1$ for each divisor $n_1$ of $n'$. Now observe that each prime divisor $q$ of $g$ is also a divisor of $d$ and hence all $dn_1$. Hence if $\psi$ is a primitive character of $C_{dn_1}$, then each summand of $\Ind_{C_{dn_1}}^{C_{dgn_1}} \psi$ is necessarily primitive so yields $g$ characters of order $dgn_1$ and we are done.
\end{proof}

\begin{corollary}
\label{genind}
Let $n>0$ be an integer, $\rho_d$ the direct sum of all characters of order $d$ of $\cc[C_d]$. Let $g=\gcd(n,d^{\infty})$. Then $$\Ind_{C_d}^{C_{dn}} \rho_d = \bigoplus\limits_{n_1| \frac{n}{g}} \frac{g\varphi(d)}{\varphi(gd)} \rho_{dgn_1}.$$
\end{corollary}

\begin{proof}
This follows from Lemma \ref{ind}, noting that $\gcd(gd,n_1)=1$ so $\varphi(gdn_1)=\varphi(gd)\varphi(n_1)$.
\end{proof}

With those lemmas out of the way, we can now start involving clusters and begin piecing the representation together. For a proper cluster $\ss$, we let $\ss^{odd}$ denote the set of odd children of $\ss$ and $n'_{\ss}$ be the orbit length of non-orphan children of $\ss$ under the inertia stabiliser $I_{\ss}$ of $\ss$.

\begin{lemma}
\label{orbsplit}
	Let $\ss$ be a proper cluster. Then $$\cc[\ss^{odd}] \ominus \mathds{1} = \left\lfloor \dfrac{|\ss^{odd}|}{n'_{\ss}} \right\rfloor \cc[C_{n'_{\ss}}] \oplus \left( |\ss^{odd}| - n'_{\ss}\left\lfloor \dfrac{|\ss^{odd}|}{n'_{\ss}} \right\rfloor -1 \right)\mathds{1}.$$
\end{lemma}

\begin{proof}
	We first claim that the number of orbits of non-orphan children of $\ss^{odd}$ under $\Is$ is $\left\lfloor \frac{|\ss^{odd}|}{n'_{\ss}} \right\rfloor$. Since the length of each orbit is $n'_{\ss}$, this is immediate if there is no orphan. Otherwise, there is exactly one orphan so the number of orbits is $\frac{|\ss^{odd}|-1}{n'_{\ss}}$. The existence of an odd orphan precludes the case $n'_{\ss}=1$ hence $\frac{|\ss^{odd}|-1}{n'_{\ss}}=\left\lfloor \frac{|\ss^{odd}|}{n'_{\ss}} \right\rfloor$. The remainder of the lemma follows since $$|\ss^{odd}| - n'_{\ss}\left\lfloor \frac{|\ss^{odd}|}{n'_{\ss}} \right\rfloor = \begin{cases}
		0 &\text{ if $\ss$ doesn't have an odd orphan,} \\
		1 &\text{ else.}
	\end{cases}$$
\end{proof}

We can now finally prove the theorem. We define $\beta(n,s)=\dfrac{\gcd(n,s^{\infty})\varphi(s)}{\varphi(s\gcd(n,s^{\infty}))}$.

\begin{proof}[Proof of Theorem \ref{repthmpf}]
	Recall that $V_{\ss} \oplus \varepsilon_{\ss}=\gamma_{\ss} \otimes (\cc[\ss^{odd}] \ominus \mathds{1})$. By Lemma \ref{orbsplit}, this equals $\gamma_{\ss} \otimes \left( \left\lfloor \dfrac{|\ss^{odd}|}{n'_{\ss}} \right\rfloor \cc[C_{n'_{\ss}}] \oplus \left( |\ss^{odd}| - n'_{\ss}\left\lfloor \dfrac{|\ss^{odd}|}{n'_{\ss}} \right\rfloor -1 \right)\mathds{1} \right)$.
	
	We first concentrate on $X= \gamma_{\ss} \otimes \cc[C_{n'_{\ss}}]$. Note that $\cc[C_{n'_{\ss}}] = \bigoplus\limits_{d \mid n_{\ss}'} \rho_d$ and hence by Lemma \ref{gentwist}, we have $X= \bigoplus\limits_{d \mid n_{\ss}'} \bigoplus\limits_{s \in S_{d,t}} \alpha_{d,t,s} \rho_s$.
	
	Now since $[I:I_{\ss}]=n_{\ss}$, $\Ind_{I_{\ss}}^I \rho_s = \Ind_{C_s}^{C_{n_{\ss}s}} \rho_s = \bigoplus\limits_{n_1 \mid \frac{n_{\ss}}{\gcd(n,s^{\infty})}} \beta(n_{\ss},s)\rho_{s\gcd(n,s^{\infty})n_1}$ by Corollary \ref{genind}. Hence $\Ind_{I_{\ss}}^I X = \bigoplus\limits_{d \mid n_{\ss}'} \bigoplus\limits_{s \in S_{d,t}} \bigoplus\limits_{n_1 \mid \frac{n_{\ss}}{\gcd(n,s^{\infty})}} \alpha_{d,t,s} \beta(n_{\ss},s)\rho_{s\gcd(n,s^{\infty})n_1}$. A similar arguments applies for the remaining terms, where we note that $\gamma_{\ss}=\frac{1}{\varphi(t)}\rho_{t}$ and use Proposition \ref{toricprop} for $\Ind_{\Is}^I \varepsilon_{\ss}$.
\end{proof}

\section{Genus two classification}
\label{tables}

Below we tabulate all possible options for cluster pictures of polynomial type (with the possible sets of depth denominators) in genus two. Moreover, we give the $H^1_{ab}$ and $H^1_t$ parts of the inertia representation for any genus two hyperelliptic curve $C:y^2=f(x)$ with that cluster picture, under the assumption that $f$ is monic. Note that it is still possible to compute $H^1_{ab}$ and $H^1_t$ without this monic assumption but we use it for the sake of simplicity (as inertia representations they only depend on the valuation of the leading coefficient). Moreover, it is always possible to make a change of variables of the curve to ensure $f$ is monic when $\deg f$ is odd.

In contrast to the 120 Namikawa--Ueno types, we find 44 cluster pictures (11 for five roots) and 276 possible denominator tuples for them (55 for five roots). We mention that there is already a type naming convention for the semistable ones (which has a correspondence to the Namikawa--Ueno types) in \cite[p.79]{DDMM}. However, we neither include nor extend this to our table since this type encodes further information such as the Frobenius action which we do not use here.



\newpage

\subsection{$|\cR|=5$}

\begin{itemize}
	\item[]
\end{itemize}
\vspace*{-2mm}

\resizebox{0.98 \textwidth}{!}{
\begin{tabular}{|c|c|c|c|c|}
\hline
\multirow{2}{*}{Cluster picture} & Possible tuples for & Congruence & \multirow{2}{*}{$H^1_{ab}$} & \multirow{2}{*}{$H^1_t$} \\
 & ($\denom d_{\cR}, \denom d_{\ss_1}, \cdots$) & conditions & & \\ \hline

\multirow{6}{*}{\clusterpicture
\Root(0.00,2)A(a1);
\Root(0.40,2)A(a2);
\Root(0.80,2)A(a3);
\Root(1.20,2)A(a4);
\Root(1.60,2)A(a5);
\ClusterL(s5){(a1)(a2)(a3)(a4)(a5)}{$\cR$};
\endclusterpicture} & \multirow{2}{*}{(1)} & $2\mid d_{\cR}$ & $4\rho_1$ & \multirow{6}{*}{0} \\
 & & $2\nmid d_{\cR}$ & $4\rho_2$ & \\ \cline{2-4}
 & (2) & - & $2\rho_4$ & \\ \cline{2-4}
 & (4) & - & $\rho_8$ &  \\ \cline{2-4}
 & \multirow{2}{*}{(5)} & $2\mid 5d_{\cR}$ & $\rho_5$ &  \\ 
 & & $2\nmid 5d_{\cR}$ & $\rho_{10}$ & \\
\hline
\end{tabular}
}

\newpage

\resizebox{0.95 \textwidth}{!}{
\begin{tabular}{|c|c|c|c|c|}
\hline
\multirow{20}{*}{\clusterpicture
\Root(0.00,2)A(a1);
\Root(0.40,2)A(a2);
\Root(1.30,2)A(a3);
\Root(1.70,2)A(a4);
\Root(2.10,2)A(a5);
\ClusterL(s1){(a1)(a2)}{$\ss_1$};
\ClusterL(s2){(a3)(a4)(a5)}{$\ss_2$};
\ClusterL(s5){(s1)(s2)}{$\cR$};
\endclusterpicture} & \multirow{4}{*}{(1,1,1)} & $2\mid d_{\cR}, 2\mid d_{\ss_2}$ & $2\rho_1$ & $\rho_1$ \\
 & & $2\mid d_{\cR}, 2\nmid d_{\ss_2}$ & $2\rho_2$ & $\rho_1$ \\
 & & $2\nmid d_{\cR}, 2\mid d_{\ss_2}$ & $2\rho_1$ & $\rho_2$ \\
 & & $2\nmid d_{\cR}, 2\nmid d_{\ss_2}$ & $2\rho_2$ & $\rho_2$ \\ \cline{2-5}
 & \multirow{2}{*}{(1,1,2)} & $2\mid d_{\cR}$ & \multirow{2}{*}{$\rho_4$} & $\rho_1$ \\
 & & $2\nmid d_{\cR}$ & & $\rho_2$ \\ \cline{2-5}
 & \multirow{4}{*}{(1,1,3)} & $2\mid d_{\cR}, 2\mid d_{\ss_2}$ & $\rho_3$ & $\rho_1$ \\
 & & $2\mid d_{\cR}, 2\nmid d_{\ss_2}$ & $\rho_6$ & $\rho_1$ \\
\hline
\end{tabular}
}

\newpage

\resizebox{1 \textwidth}{!}{
\begin{tabular}{|c|c|c|c|c|}
\hline
\multirow{6}{*}{\clusterpicture
\Root(0.00,2)A(a1);
\Root(0.40,2)A(a2);
\Root(1.00,2)A(a3);
\Root(1.40,2)A(a4);
\Root(2.00,2)A(a5);
\ClusterL(s1){(a1)(a2)}{$\ss_2$};
\ClusterL(s2){(s1)(a3)(a4)}{$\ss_1$};
\ClusterL(s5){(s1)(s2)(a5)}{$\cR$};
\endclusterpicture} & \multirow{2}{*}{(1,1,1)} & $2\mid d_{\cR}$ & \multirow{6}{*}{0} & $2\rho_1$ \\
 & & $2\nmid d_{\cR}$ & & $2\rho_2$ \\ \cline{2-3}\cline{5-5}
 & \multirow{2}{*}{(1,1,2)} & $2\mid d_{\cR}$ & & $2\rho_1$ \\
 & & $2\nmid d_{\cR}$ & & $2\rho_2$ \\ \cline{2-3}\cline{5-5}
 & (1,2,1) & - & & $\rho_1 \oplus \rho_2$ \\ \cline{2-3}\cline{5-5}
 & (1,2,2) & - & & $\rho_1 \oplus \rho_2$ \\ \hline

\multirow{12}{*}{\clusterpicture
\Root(0.00,2)A(a1);
\Root(0.40,2)A(a2);
\Root(1.00,2)A(a3);
\Root(1.80,2)A(a4);
\Root(2.20,2)A(a5);
\ClusterL(s1){(a1)(a2)}{$\ss_3$};
\ClusterL(s2){(s1)(a3)}{$\ss_2$};
\ClusterL(s3){(a4)(a5)}{$\ss_1$};
\ClusterL(s5){(s2)(s3)}{$\cR$};
\endclusterpicture} & \multirow{3}{*}{(1,1,1,1)} & $2\mid d_{\cR},2\mid d_{\ss_2}$ & \multirow{12}{*}{0} & $2\rho_1$ \\
 & & $2\nmid d_{\cR},2\nmid d_{\ss_2}$ & & $2\rho_2$ \\
\hline
\end{tabular}
}

\newpage

\subsection{$|\cR|=6$}

\begin{itemize}
	\item[]
\end{itemize}
\vspace*{-2mm}


\bibliographystyle{alpha}
\bibliography{Everything}
\end{document}